\documentclass[final,3p,times]{elsarticle}

\usepackage{amsmath,amssymb,amsthm}
\usepackage{bm,multirow}

\numberwithin{equation}{section}

\theoremstyle{plain}
\newtheorem{theorem}{Theorem}[section]
\newtheorem{proposition}[theorem]{Proposition}
\newtheorem{lemma}[theorem]{Lemma}
\newtheorem{corollary}[theorem]{Corollary}

\theoremstyle{remark}
\newtheorem{remark}[theorem]{Remark}
\newtheorem{assumption}[theorem]{Assumption}

\theoremstyle{definition}

\newcommand{\gap}{\hspace{0.1cm}}
\newcommand{\vertiii}[1]{{\left\vert\kern-0.25ex\left\vert\kern-0.25ex\left\vert #1 \right\vert\kern-0.25ex\right\vert\kern-0.25ex\right\vert}}

\newcommand{\bb}{\mathbf{b}}
\newcommand{\bn}{\mathbf{n}}
\newcommand{\bt}{\mathbf{t}}
\newcommand{\bu}{\mathbf{u}}
\newcommand{\bv}{\mathbf{v}}
\newcommand{\bx}{\mathbf{x}}

\newcommand{\bB}{\mathbf{B}}
\newcommand{\bD}{\mathbf{D}}

\newcommand{\bV}{\mathbf{V}}
\newcommand{\bdelta}{\bm{\delta}}
\newcommand{\bepsilon}{\bm{\epsilon}}
\newcommand{\bsigma}{\bm{\sigma}}
\newcommand{\btau}{\bm{\tau}}

\newcommand{\cD}{\mathcal{D}}
\newcommand{\cT}{\mathcal{T}}
\newcommand{\cP}{\mathcal{P}}
\newcommand{\cQ}{\mathcal{Q}}

\newcommand{\intO}{\int_{\Omega}}

\let\div\relax
\DeclareMathOperator{\div}{div}
\DeclareMathOperator{\proj}{proj}

\usepackage[normalem]{ulem}
\usepackage{color}

\journal{arXiv}

\begin{document}

\begin{frontmatter}
\title{A variational framework for the strain-smoothed element method\tnoteref{acknowledgement}}
\tnotetext[acknowledgement]{
Chaemin Lee's work was supported by the BK21 FOUR program through the National Research Foundation of Korea~(NRF) funded by the Ministry of Education, and
Jongho Park's work was supported by Basic Science Research Program through NRF funded by the Ministry of Education~(2019R1A6A1A10073887). }

\author[CL]{Chaemin Lee}
\address[CL]{Department of Mechanical Engineering, KAIST, Daejeon 34141, Korea}
\ead{ghi9000@kaist.ac.kr}

\author[JP]{Jongho Park\corref{cor}}
\ead{jongho.park@kaist.ac.kr}
\ead[url]{https://sites.google.com/view/jonghopark}
\address[JP]{Natural Science Research Institute, KAIST, Daejeon 34141, Korea}

\cortext[cor]{Corresponding author}

\begin{abstract}
This paper is devoted to a rigorous mathematical foundation for the convergence properties of the strain-smoothed element~(SSE) method.
The SSE method has demonstrated improved convergence behaviors compared to other strain smoothing methods through various numerical examples; however, there has been no theoretical evidence for the convergence behavior.
A unique feature of the SSE method is the construction of smoothed strain fields within elements by fully unifying the strains of adjacent elements.
Owing to this feature, convergence analysis is required, which is different from other existing strain smoothing methods.
In this paper, we first propose a novel mixed variational principle wherein the SSE method can be interpreted as a Galerkin approximation of that.
The proposed variational principle is a generalization of the well-known Hu--Washizu variational principle; thus, various existing strain smoothing methods can be expressed in terms of the proposed variational principle.
With a unified view of the SSE method and other existing methods through the proposed variational principle, we analyze the convergence behavior of the SSE method and explain the reason for the improved performance compared to other methods.
We also present numerical experiments that support our theoretical results.
\end{abstract}


\end{frontmatter}


\section{Introduction}
\label{Sec:Introduction}
The finite element method~(FEM) has developed into one of the most powerful numerical methods for solving problems in engineering and mathematical models. 
The method can solve many important physical problems such as solid mechanics, fluid dynamics, heat transfer, and multi-physics problems. 
For several decades, substantial efforts have been made to develop low-order finite elements exhibiting high accuracy in coarse meshes.
Low-order elements have high modeling capabilities and are particularly preferred for large deformation analyses requiring automatic remeshing. 
In addition, they often provide a relatively easy way to solve complicated engineering problems such as contact analysis~\cite{Hughes:2000,Bathe:1996,LN:2010}. 

There have been numerous attempts to develop more effective finite element methods. 
One major attempt is the assumed strain methods wherein the standard discrete gradient operator is replaced with an assumed form~\cite{Hughes:1980,STP:1985,BB:1986,WI:1990,IW:1991}. 
These methods effectively alleviate locking in finite elements and can be formulated within the framework of the Hu--Washizu variational principle~\cite{SH:1986}. 
The partition of unity finite element method~(PU-FEM)~\cite{MB:1996,BM:1997}, generalized FEM~(G-FEM)~\cite{SBC:2000}, and extended FEM~(X-FEM)~\cite{BB:1999,MDB:1999} are also good examples. 
These methods extend the approximation space by introducing special enrichment functions associated with the problem of interest and additional degrees of freedom. 

We can improve the performance of finite elements without using additional degrees of freedom through strain smoothing. 
The strain smoothing technique was first proposed by Chen et al.~\cite{CWYY:2001}. 
Subsequently, it was extended to a generalized form allowing for discontinuous displacement functions~\cite{Liu:2008}, and this formed the theoretical basis for the smoothed point interpolation method~(S-PIM)~\cite{LZDWZLH:2005,LJCZZ:2011,LLY:2018,LL2:2019,YCL:2020}. 
Liu et al.\ proposed a series of smoothed finite element methods~(S-FEMs) with different smoothing domains~\cite{LDN:2007,LNNL:2009,LNL:2009,NVRN:2010,VNC:2013,NBO:2015,LAHKB:2017,FOBN:2017,NPNK:2016,KLKI:2018,SHCI:2013,JSI:2016,NNCVN:2016,CNN:2017,YWZJF:2019,JYYC:2020,HMX:2020}.      
The smoothing domains can be configured based on edges, nodes, and cells, and piecewise constant strain fields are constructed for the smoothing domains. 
Each of the S-FEMs has attractive properties.
The edge-based S-FEM~(ES-FEM) generally shows the best convergence behavior among them~\cite{LNL:2009}, the node-based S-FEM~(NS-FEM) alleviates volumetric locking~\cite{LNNL:2009}, and the cell-based S-FEM~(CS-FEM) has been applied in various studies because of its convenience and effectiveness~\cite{LDN:2007,NBO:2015,KLKI:2018,SHCI:2013,JSI:2016}.
Theoretical studies on S-FEMs were conducted, and a variational framework was established based on either the Hellinger--Reissner or Hu--Washizu variational principle~\cite{LN:2010,NBN:2008,LNN:2010}.
One may refer to~\cite{ZL:2018} for a recent survey on the class of S-FEMs. 

Recently, a new strain smoothing method known as the strain-smoothed element (SSE) method was developed to improve the predictive capability of low-order elements~\cite{LL:2018,LL:2019,LKL:2021}. 
Using the SSE method, smoothed strain fields are constructed for elements, not for the smoothing domains, and strains of all adjacent elements are fully utilized for strain smoothing. 
Therefore, the SSE method has an advantage over other existing methods in that it does not require additional grids for implementation.
It has been successfully applied to 3-node triangular and 4-node tetrahedral solid elements~\cite{LL:2018}, 4-node quadrilateral solid element~\cite{LKL:2021} and the 3-node mixed interpolation of tensorial components~(MITC3+) shell element~\cite{LL:2019}.
Various numerical experiments have verified that the strain-smoothed elements yield highly accurate solutions compared with other competitive elements.

Till date, the properties of the SSE method have only been numerically verified.
This paper is devoted to the theoretical aspects of the SSE method: convergence analysis, underlying variational principle, and an explanation for faster convergence compared to conventional FEMs.
First, we note that there have been many fruitful results for theoretical studies on strain smoothing methods~\cite{LN:2010,NBN:2008,LNN:2010,Liu:2009,Liu:2010,CLL:2016}.
A strain smoothing method with stabilized conforming nodal integration was analyzed based on the Hu--Washizu variational principle in~\cite{NBN:2008}.
In~\cite{LNN:2010}, the strain smoothing procedure of S-FEMs is shown as an orthogonal projection between the assumed strain spaces.
Using this fact, S-FEMs are analyzed in terms of the Hellinger--Reissner variational framework.
In~\cite{Liu:2009,Liu:2010,CLL:2016}, the G space theory that presents a unified analysis for various strain smoothing procedures was established.
However, the abovementioned results are not directly applicable to the SSE method owing to the rather complicated structure of the method.
The smoothed strain field of the SSE method is constructed by a particular strain smoothing followed by additional pointwise assignment to Gaussian points and interpolation within an element~\cite{LL:2018,LL:2019,LKL:2021}.
Owing to its features, the method can neither be analyzed in terms of the Hu--Washizu variational principle as in~\cite{NBN:2008} nor  be interpreted in the G space theory.
In particular, the strain smoothing step of the SSE method cannot be expressed as an orthogonal projection between the assumed strain spaces, as in~\cite{LNN:2010}.
Therefore, a new theory is required to explain the convergence of the SSE method.

We first observe that the SSE method can be cast into an equivalent formulation whose strain smoothing step is a composition of orthogonal projection operators among the assumed strain spaces.
More precisely, we demonstrate that the smoothed strain of the method can be obtained by applying a sequence of orthogonal projection operators from the assumed strain spaces corresponding to coarser meshes to those corresponding to finer meshes.
By invoking this observation, we herein construct a mixed variational principle that can derive the SSE method as a conforming Galerkin approximation.
The constructed variational principle naturally generalizes the Hu--Washizu variational principle so that it can provide a unified convergence analysis of the standard FEM, the S-FEM, and the SSE method. 
Applying the standard convergence theory for mixed FEMs~\cite{BS:2008,BBF:2013} to the constructed variational principle yields a unifying convergence theorem for these methods, and the improved performance of the SSE method compared to other methods can be explained through the unifying theorem.
Indeed, we show the following:
\begin{itemize}
\item All above methods can be represented as conforming discretizations of the proposed variational principle, and the SSE method uses finer grids than the others.
\item The strain error bound of the SSE method is $O(h)$, where $h$ stands for the maximum element diameter.
\end{itemize}
Some numerical experiments are conducted to support the presented theoretical properties.
While we herein deal with the 3-node triangular element~\cite{LL:2018} and the 4-node quadrilateral element~\cite{LKL:2021}, our argument herein can be generalized straightforwardly to polygonal elements.

The remainder of this paper is organized as follows.
The displacement variational formulation for linear elasticity is reviewed in Sect.~\ref{Sec:Elasticity}.
In Sect.~\ref{Sec:SSE}, we introduce the SSE method and show that the method can be interpreted from the viewpoint of projection operators.
The variational framework for the SSE method is established in Sect.~\ref{Sec:VP}.
In Sect.~\ref{Sec:Convergence}, the convergence theory for the SSE method based on the variational principle established in Sect.~\ref{Sec:VP} is presented.
Several numerical results that support our theory are provided in Sect.~\ref{Sec:Numerical}.
We provide concluding statements in Sect.~\ref{Sec:Conclusion}.

\section{Linear elasticity}
\label{Sec:Elasticity}
\label{Sec:Elasticity}
We consider a linear elastic problem.
Let $\Omega \subset \mathbb{R}^2$ be a bounded and polygonal domain representing a two-dimensional linear elastic solid.
The boundary $\partial \Omega$ of $\Omega$ comprises two parts: $\Gamma_D \neq \emptyset$ and $\Gamma_N = \partial \Omega \setminus \Gamma_D$.
The equilibrium equation is stated as
\begin{equation}
\label{governing_eq}
\div \bsigma + \bb = \mathbf{0} \gap\textrm{ in } \Omega
\end{equation}
with the Dirichlet boundary condition
\begin{equation}
\label{Dirichlet}
\bu = \bu_{\Gamma} \gap \textrm{ on } \Gamma_D
\end{equation}
and the Neumann boundary condition
\begin{equation}
\label{Neumann}
\bsigma \bn = \bt \gap\textrm{ on } \Gamma_N,
\end{equation}
where $\bsigma$ is the Cauchy stress, $\bu$ is the displacement field, $\bb$ is the body force, $\bu_{\Gamma}$ is the prescribed displacement on $\Gamma_D$, $\bt$ is the prescribed traction on $\Gamma_N$, and $\bn$ is the unit outward normal to $\Gamma_N$.
To simplify the presentation, we introduce the Voigt notation for stress and strain, i.e., stress and strain tensors are written as column vectors:
\begin{equation*}
\bsigma = \begin{bmatrix}\sigma_{xx}& \sigma_{yy}& \sigma_{xy} \end{bmatrix}^T, \quad
\bepsilon = \begin{bmatrix} \epsilon_{xx} & \epsilon_{yy} & 2\epsilon_{xy} \end{bmatrix}^T.
\end{equation*}
Subsequently, the compatibility relation between the displacement $\bu$ and the strain $\bepsilon$ is expressed as
\begin{equation}
\label{governing_comp}
\bepsilon = \bB \bu \gap\textrm{ in } \Omega,
\end{equation}
where $\bB$ is a matrix of differential operators given by
\begin{equation*}
\bB = \begin{bmatrix}\frac{\partial}{\partial x} & 0 & \frac{\partial}{\partial y}\\
0 & \frac{\partial}{\partial y} & \frac{\partial}{\partial x} \end{bmatrix}^T .
\end{equation*}
The stress-strain constitutive equation is written as follows:
\begin{equation}
\label{governing_cons}
\bsigma = \bD \bepsilon \gap\textrm{ in } \Omega,
\end{equation}
where $\bD$ is a $3 \times 3$ symmetric and positive definite matrix that relies on a material composed of an elastic solid.
We assume that the material is uniform, i.e., $\bD$ is constant in $\Omega$.
The linear elastic problem is governed by three equations~\eqref{governing_eq},~\eqref{governing_comp}, and~\eqref{governing_cons} with the boundary conditions~\eqref{Dirichlet} and~\eqref{Neumann}.

Next, we consider the weak formulation, i.e., the displacement variational formulation for the linear elastic problem.
In the following, we set $\bu_{\Gamma} = \mathbf{0}$ in~\eqref{Dirichlet} for simplicity.
Let $V$ be a space of kinematically admissible displacement fields defined as
\begin{equation*}
V = \left\{ \bu \in (H^1 (\Omega))^2 : \bu = \mathbf{0} \textrm{ on } \Gamma_D \right\},
\end{equation*}
where $H^k(\Omega)$, $k \geq 1$,  is the collection of $L^2 (\Omega)$-functions whose all $k$th order partial derivatives are in $L^2 (\Omega)$.
A space $W$ of strain and stress fields is given by
\begin{equation*}
W = (L^2 (\Omega))^3.
\end{equation*}
A bilinear form $a(\cdot, \cdot)$ on $V$ is defined by
\begin{equation}
\label{bilinear}
a(\bu, \bv) = \intO \bD \bepsilon[\bu] : \bepsilon[\bv] \,d\Omega, \quad \bu, \bv \in V,
\end{equation}
where $\bepsilon [\bu] = \bB \bu$, and the symbol $:$ denotes the Euclidean inner product in $\mathbb{R}^3$.
Note that for $\bu \in V$, we have $\bepsilon [\bu] \in W$.
Clearly, $a(\cdot, \cdot)$ is symmetric, continuous, and coercive~\cite[Chapter~11]{BS:2008}.
Let $f$ denote a continuous linear functional on $V$ given by
\begin{equation*}
f (\bu) = \intO \bb \cdot \bu \,d \Omega + \int_{\Gamma_N} \bt \cdot \bu \, d \Gamma, \quad \bu \in V.
\end{equation*}
It is well-known that, under very mild conditions on $\bb$ and $\bt$~(see, e.g.,~\cite{BS:2008}), a solution of the linear elastic problem is in $(H^2 (\Omega))^2$ and is characterized by the following variational problem: find $\bu \in V$ such that
\begin{equation}
\label{weak}
a( \bu, \bv) = f ( \bv ) \quad \forall \bv \in V.
\end{equation}
By the Lax--Milgram theorem~\cite[Theorem~2.7.7]{BS:2008}, the problem~\eqref{weak} has a unique solution and it solves the following quadratic optimization problem:
\begin{equation}
\label{weak_min}
\min_{\bu \in V} \left\{ \frac{1}{2} a(\bu, \bu) - f ( \bu) \right\}.
\end{equation}

\section{The strain-smoothed element method}
\label{Sec:SSE}
In this section, we briefly introduce the SSE method for solving~\eqref{weak}.
We closely follow the explanations presented in~\cite{LL:2018,LKL:2021}.
In addition, we present an alternative view to the SSE method that can be described in terms of orthogonal projection operators defined on particular meshes.
Similar discussions were made in~\cite{LNN:2010} for the S-FEMs.

For a subregion $K$ of $\Omega$ and a nonnegative integer $n$, let $\cP_n (K)$ denote the collection of all polynomials of degree less than or equal to $n$ on $K$.

\subsection{Strain-smoothed 3-node triangular element}
\label{Subsec:SSE3}
We describe the strain-smoothed 3-node triangular element proposed in~\cite{LL:2018}.
Let $\cT_h$ be a quasi-uniform triangulation of $\Omega$ with a maximum element diameter $h > 0$.
We set the discrete displacement space $V_h \subset V$ as the collection of the continuous and piecewise linear functions on $\cT_h$ satisfying the homogeneous Dirichlet boundary condition on $\Gamma_D$, i.e.,
\begin{equation*}
V_h = \left\{ \bu \in V : \bu|_{T} \in (\cP_1 (T))^2 \gap\forall T \in \cT_h \right\}.
\end{equation*}
We define the discrete strain--stress space $W_h$ associated with the subdivision $\cT_h$ as
\begin{equation*}
W_{h} = \left\{ \bepsilon \in W : \bepsilon|_{T} \in (\cP_0 (T))^3 \gap\forall T \in \cT_{h} \right\}.
\end{equation*}
Clearly, $\bepsilon [\bu] = \bB \bu$ and $\bsigma [\bu] = \bD \bB \bu$ belong to $W_h$ when $\bu \in V_h$.

The standard FEM for linear elasticity solves the Galerkin approximation of~\eqref{weak} defined on $V_h$: find $\bu_h \in V_h$ such that
\begin{equation*}
a(\bu_h, \bv) = f (\bv) \quad \forall \bv \in V_h,
\end{equation*}
where the bilinear form $a(\cdot, \cdot)$:~$V_h \times V_h \rightarrow \mathbb{R}$ was given in~\eqref{bilinear}.
For the SSE method~\cite{LL:2018}, we use an alternative bilinear form $\bar{a}(\cdot, \cdot)$:~$V_h \times V_h \rightarrow \mathbb{R}$ by replacing $\bepsilon [\bu]$ in~\eqref{bilinear} with an appropriate \textit{smoothed} strain field $\bar{\bepsilon}[\bu]$, that is,
\begin{equation}
\label{bilinear_smoothed}
\bar{a}(\bu, \bv) = \intO \bD \bar{\bepsilon}[\bu]: \bar{\bepsilon}[\bv] \,d\Omega, \quad \bu, \bv \in V_h.
\end{equation}

\begin{figure}[]
\centering
\includegraphics[width=0.8\textwidth]{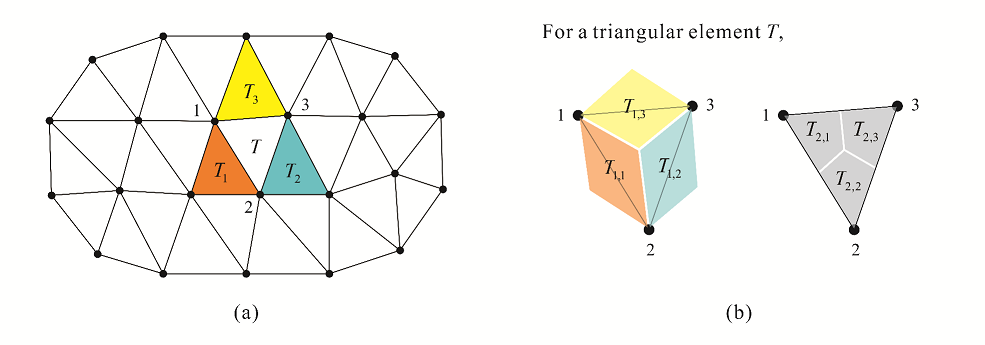}
\caption{(a) Three neighboring elements $T_1$, $T_2$, and $T_3$ of an interior element $T \in \cT_h$. (b) $T_{1,i}$ and $T_{2,i}$, $i=1,2,3$ are the subregions in $\cT_{1,h}$ and $\cT_{2,h}$ that overlap with $T$, respectively.}
\label{Fig:adjacent}
\end{figure}

\begin{figure}[]
\centering
\includegraphics[width=0.8\textwidth]{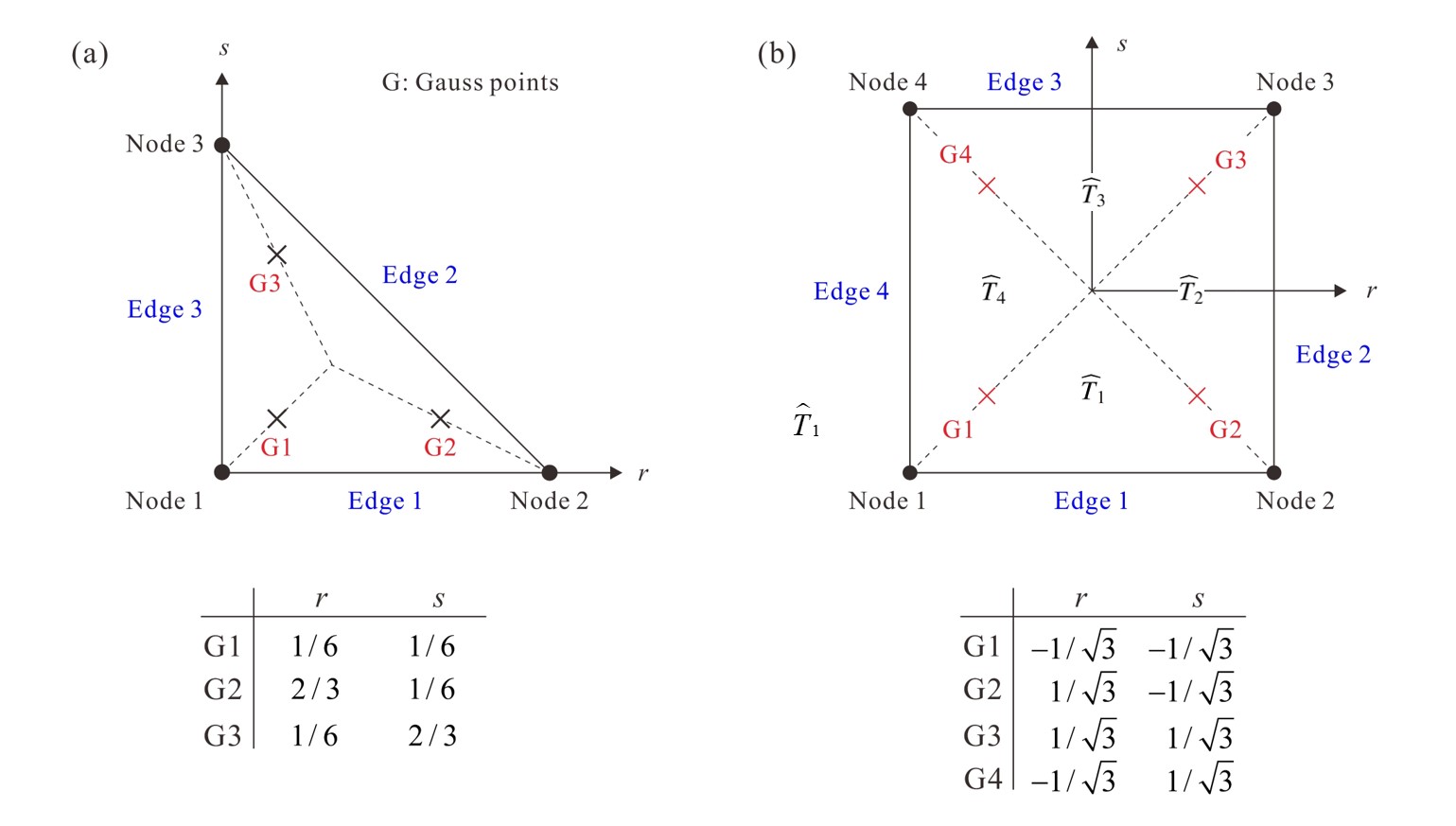}
\caption{Node and edge numbering conventions for (a) triangular elements and (b) quadrilateral elements.
Gaussian integration points for the elements are denoted by $\mathrm{G}$, and four subtriangles $\widehat{T}_1$, $\widehat{T}_2$, $\widehat{T}_3$, and $\widehat{T}_4$ are defined for the quadrilateral element.}
\label{Fig:coord}
\end{figure}

In the following, we present the construction of the SSE smoothing operator $S_h$:~$W_h \rightarrow \overline{W}_h$ that maps a given strain field $\bepsilon \in W_h$ to the corresponding smoothed strain field $\bar{\bepsilon} \in \overline{W}_h$, where
\begin{equation*}
\overline{W}_{h} = \left\{ \bar{\bepsilon} \in W : \bar{\bepsilon}|_{T} \in (\cP_1 (T))^3 \gap\forall T \in \cT_{h} \right\}. 
\end{equation*}
That is, the resulting $\bar{\bepsilon} = S_h \bepsilon$ is piecewise linear.
Take any element $T \in \cT_h$.
We first assume that $T$ is an interior element; that is, there exist three elements $T_1$, $T_2$, and $T_3$ in $\cT_{h}$ adjacent to $T$, as shown in Fig.~\ref{Fig:adjacent}(a).
Intermediate smoothed strains $\hat{\bepsilon}^{(i)} \in \mathbb{R}^3$, $i=1,2,3$, are defined by
\begin{equation}
\label{smoothed_intermediate}
\hat{\bepsilon}^{(i)} = \frac{1}{|T| + |T_i|} \left( \int_{T} \bepsilon \,d\Omega + \int_{T_i} \bepsilon \,d\Omega \right).
\end{equation}
If $T$ is an exterior element, that is, there is no adjacent element $T_i$ for some $i$, then the corresponding intermediate smoothed strain $\hat{\bepsilon}^{(i)}$ is defined by simply replacing $T_i$ with~\eqref{smoothed_intermediate} by $T$.
Differently from the existing strain smoothing methods~(see, e.g.,~\cite{LN:2010,LDN:2007,NBN:2008,ZL:2018}), the SSE method has a procedure for unifying the piecewise constant intermediate smoothed strains $\hat{\bepsilon}^{(i)}$ within the element, thereby resulting in a linear smoothed strain field $\bar{\bepsilon}$.
Using the intermediate smoothed strains in~\eqref{smoothed_intermediate}, we assign the pointwise values of $\bar{\bepsilon}$ at three Gaussian integration points~($\mathrm{G}1$, $\mathrm{G}2$, and $\mathrm{G}3$ in Fig.~\ref{Fig:coord}(a)) of $T$ in the following manner:
\begin{equation}
\label{Gaussian}
\bar{\bepsilon} (\mathrm{G}i) = \frac{1}{2} (\hat{\bepsilon}^{(i-1)} + \hat{\bepsilon}^{(i)})
\end{equation}
with the convention $\hat{\bepsilon}^{(0)} = \hat{\bepsilon}^{(3)}$, where $i=1,2,3$.
From~\eqref{Gaussian}, the smoothed strain field $\bar{\bepsilon}$ in~\eqref{bilinear_smoothed} is uniquely determined for $T$ by linear interpolation.

Finally, we have
\begin{equation}
\label{bilinear_SSE1}
\bar{a}(\bu, \bv) = \intO \bD S_h \bepsilon [\bu]: S_h \bepsilon [\bv] \,d\Omega, \quad \bu, \bv \in V_h
\end{equation}
and solve the following problem: find $\bar{\bu}_h \in V_h$ such that
\begin{equation}
\label{weak_SSE}
\bar{a}(\bar{\bu}_h, \bv) = f ( \bv) \quad \forall \bv \in V_h.
\end{equation}

\subsection{Strain-smoothed 4-node quadrilateral element}
\label{Subsec:SSE4}
Recently, the SSE method has been extended for the 3-node MITC shell element and the 4-node quadrilateral solid element~\cite{LL:2019,LKL:2021}.  
Here, we deal with the strain-smoothed 4-node quadrilateral element~\cite{LKL:2021}.
Let $\cT_h$ be a quasi-uniform subdivision of the domain $\Omega$ consisting of quadrilateral elements with a maximum element diameter $h > 0$.
Subsequently, a triangulation $\widehat{\cT}_h$ of $\Omega$ is formed by subdividing each element in $\cT_h$ into four nonoverlapping subtriangles based on the nodes and the center point~($r=s=0$~in Fig.~\ref{Fig:coord}(b)) of the element.
The discrete displacement space for the 4-node quadrilateral element using piecewise linear shape functions proposed in~\cite{KL:2018} on $\cT_h$ is given by
\begin{equation*}
V_h = \left\{ \bu \in V : \bu|_{\widehat{T}} \in (\cP_1 (\widehat{T}))^2 \gap\forall \widehat{T} \in \widehat{\cT}_h
\textrm{ and } \bu (\bx_{T,0}) = \frac{1}{4} \sum_{i=1}^4 \bu (\bx_{T,i}) \gap\forall T \in \cT_h \right\},
\end{equation*}
where $\bx_{T,0}$ is the center point of the element $T \in \cT_h$ and $\bx_{T,i}$, $i=1,2,3,4$, are the nodes of $T$.
Then, for any $\bu \in V_h$, $\bepsilon [\bu] = \bB \bu \in W_h$ and $\bsigma [\bu] = \bD \bB \bu \in W_h$ belong to $W_h$, where
\begin{equation*}
W_h = \left\{ \bepsilon \in W: \bepsilon|_{\widehat{T}} \in (\cP_0 (\widehat{T}))^3 \gap\forall \widehat{T} \in \widehat{\cT}_h \right\}.
\end{equation*}
For the strain field $\bepsilon \in W_h$, the corresponding smoothed strain field $\bar{\bepsilon} = S_h \bepsilon$ is contained in
\begin{equation*}
\overline{W}_h = \left\{ \bar{\bepsilon} \in W : \bepsilon|_T \in (\cQ_{1,1}(T))^3 \gap\forall T \in \cT_h \right\},
\end{equation*}
where $\cQ_{1,1}(T)$ denotes the collection of all bilinear functions on $T$.
We take any element $T \in \cT_h$.
The element $T$ consists of four subtriangles $\widehat{T}_1$, $\widehat{T}_2$, $\widehat{T}_3$, and $\widehat{T}_4$ in $\widehat{\cT}_h$, as depicted in Fig.~\ref{Fig:coord}(b).
If $T$ is an interior element, then for each of $\widehat{T}_i$, $i=1,2,3,4$, it has a neighboring subtriangle $\widehat{T}_{i}^*$ that belongs to an element adjacent to $T$.
Intermediate smoothed strains $\hat{\bepsilon}^{(i)} \in \mathbb{R}^3$, $i=1,2,3,4$, are defined by
\begin{equation}
\label{smoothed_intermediate_quad}
\hat{\bepsilon}^{(i)}= \frac{1}{|\widehat{T}_i| + |\widehat{T}_i^*|} \left( \int_{\widehat{T}_i} \bepsilon \,d\Omega + \int_{\widehat{T}_i^*} \bepsilon \,d\Omega \right).
\end{equation}
If $T$ is an exterior element such that there does not exist a neighboring subtriangle $\widehat{T}_{i}^*$ for some $i$, then the corresponding intermediate smoothed strain $\hat{\bepsilon}^{(i)}$ is defined by replacing $\widehat{T}_i^*$ in~\eqref{smoothed_intermediate_quad} with $\widehat{T}_i$.
Similar to the case of 3-node triangular element, there is an additional unification step in the SSE method.
The pointwise values of $\bar{\bepsilon}$ at the four Gaussian integration points~(see Fig.~\ref{Fig:coord}(b)) are determined by
\begin{equation*}
\bar{\bepsilon} (\mathrm{G}i) = \frac{1}{|\widehat{T}_{i-1}| + |\widehat{T}_i|}\left( |\widehat{T}_{i-1}| \hat{\bepsilon}^{(i-1)} + |\widehat{T}_{i}| \hat{\bepsilon}^{(i)} \right),
\end{equation*}
with the conventions $\widehat{T}_0 = \widehat{T}_4$ and $\hat{\bepsilon}^{(0)} = \hat{\bepsilon}^{(4)}$.
Finally, the smoothed strain field $\bar {\bepsilon}$ is determined on $T$ by bilinear interpolation.
We solve the variational problem~\eqref{weak_SSE} with the smoothed strain field define as above.

\subsection{An alternative view: twice-projected strain}
\begin{figure}[]
\centering
\includegraphics[width=0.8\textwidth]{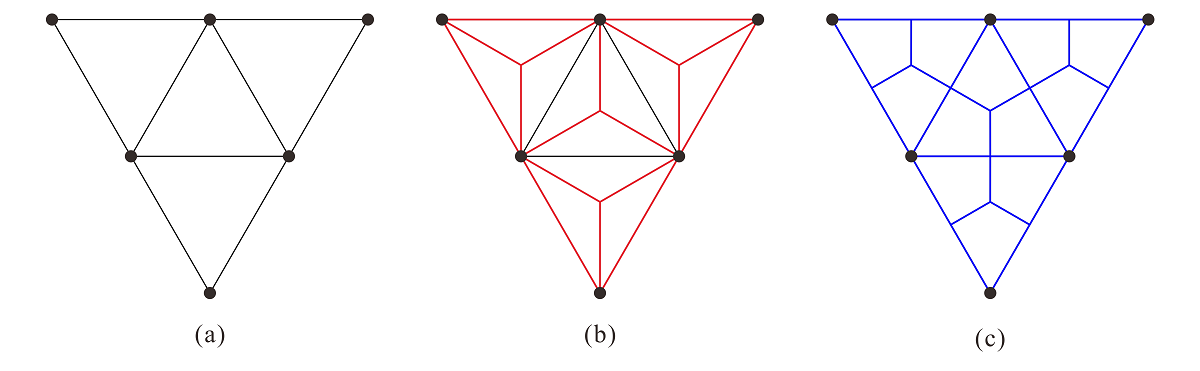}
\caption{Three subdivisions of the domain $\Omega$ for the strain-smoothed 3-node triangular element: (a) $\cT_h$, (b) $\cT_{1,h}$, and (c) $\cT_{2,h}$.}
\label{Fig:subdivision}
\end{figure}

We present an alternative derivation of the SSE method, which will be useful in the convergence analysis of the method.
An alternative smoothed strain field $\bar{\bepsilon}$ defined in the following is different from that explained above, but it eventually provides an equivalent formulation to~\eqref{weak_SSE}.

First, we consider the strain-smoothed 3-node triangular element introduced in Sect.~\ref{Subsec:SSE3}.
We construct two subdivisions $\cT_{1,h}$ and $\cT_{2,h}$ of $\Omega$, other than $\cT_h$, as follows.
For two neighboring elements $T_1$ and $T_2$ in $\cT_h$, let $e$ be the edge shared by them.
Subsequently, we consider a quadrilateral whose vertices are the endpoints of $e$ and the centroids of $T_1$ and $T_2$.
We define $\cT_{1,h}$ as the collection of such quadrilaterals.
In order to construct $\cT_{2,h}$, we partition each element of $\cT_h$ into three pieces by joining the centroid and the midpoints of the element edges.
Then, $\cT_{2,h}$ is defined as the collection of such pieces.
Fig.~\ref{Fig:subdivision} displays $\cT_h$, $\cT_{1,h}$, and $\cT_{2,h}$.

In the case of the strain-smoothed 4-node quadrilateral element introduced in Sect.~\ref{Subsec:SSE4}, subdivisions $\cT_{1,h}$ and $\cT_{2,h}$ can be defined in an analogous manner.
More precisely, $\cT_{1,h}$ is the collection of quadrilaterals whose vertices are the center points of each of the two adjacent elements in $\cT_h$ and the endpoints of the shared edge.
On the contrary, $\cT_{2,h}$ consists of quadrilaterals formed by joining the center point and the midpoints of the edges of each element in $\cT_h$.
In what follows, we deal with the 3-node triangular element and the 4-node quadrilateral element in a unified fashion.

For $k=1,2$, let $W_{k,h} \subset W$ be the collection of piecewise constant functions on $\cT_{k,h}$, i.e.,
\begin{equation*}
W_{k,h} = \left\{ \bepsilon \in W : \bepsilon|_{T} \in (\cP_0 (T))^3 \gap\forall T \in \cT_{k,h} \right\}.
\end{equation*}
The piecewise smoothing operator $P_{k,h}$:~$W \rightarrow W_{k,h}$ is defined by
\begin{equation}
\label{P_h}
(P_{k,h}\bepsilon)(x) = \frac{1}{|T|} \int_{T} \bepsilon \,d\Omega, \quad \bepsilon \in W,\gap T \in \cT_{k,h},\gap x \in T. 
\end{equation}
It was observed in~\cite{LNN:2010} that piecewise smoothing operators of the form~\eqref{P_h} are orthogonal projectors; rigorous statements are provided in the following lemmas.

\begin{lemma}
\label{Lem:commute}
Let $\mathbf{A}$ be a $3 \times 3$ matrix.
For $k=1,2$, the piecewise smoothing operator $P_{k,h}$ commutes with $\mathbf{A}$, i.e.,
\begin{equation*}
P_{k,h}(\mathbf{A}\bepsilon) = \mathbf{A}P_{k,h} \bepsilon, \quad \bepsilon \in W.
\end{equation*}
\end{lemma}
\begin{proof}
It is elementary.
\end{proof}

\begin{lemma}
\label{Lem:orthogonal}
For $k=1,2$, the piecewise smoothing operator $P_{k,h}$ is the $(L^2(\Omega))^3$-orthogonal projection onto $W_{k,h}$, i.e., $P_{k,h}^2 = P_{k,h}$ and
\begin{equation*}
\intO P_{k,h} \bepsilon : \bdelta \,d\Omega = \intO \bepsilon : P_{k,h} \bdelta \,d\Omega, \quad \bepsilon, \bdelta \in W.
\end{equation*}
\end{lemma}
\begin{proof}
See~\cite[Remarks~2 and~4]{LNN:2010}.
\end{proof}

Now, we set $\bar{\bepsilon} = P_{2,h} P_{1,h}\bepsilon$ in~\eqref{bilinear_smoothed}.
That is, we have
\begin{equation}
\label{bilinear_SSE2}
\bar{a}(\bu, \bv) = \intO \bD P_{2,h} P_{1,h} \bepsilon [\bu]: P_{2,h} P_{1,h} \bepsilon [\bv] \,d\Omega, \quad \bu, \bv \in V_h .
\end{equation}
We note that $\bar{\bepsilon} = P_{2,h} P_{1,h} \bepsilon \in W_{2,h}$ in~\eqref{bilinear_SSE2}, whereas its counterpart $\bar{\bepsilon} = S_h \bepsilon$ in~\eqref{bilinear_SSE1} belongs to $\overline{W}_h$.
Even though~\eqref{bilinear_SSE1} and~\eqref{bilinear_SSE2} use different smoothed strain fields, one can prove that they result in the same bilinear form $\bar{a}(\cdot, \cdot)$. 

\begin{figure}[]
\centering
\includegraphics[width=0.8\textwidth]{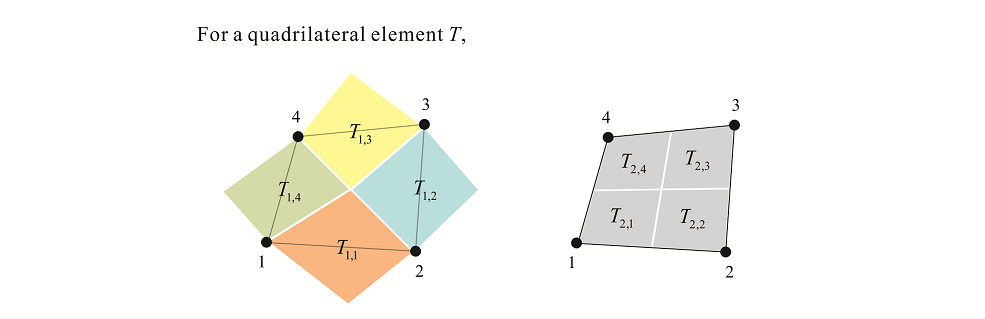}
\caption{Subregions $T_{1,i} \in \cT_{1,h}$ and $T_{2,i} \in \cT_{2,h}$, $i=1,2,3,4$, for the strain-smoothed 4-node quadrilateral element.}
\label{Fig:quad}
\end{figure}

\begin{theorem}
\label{Thm:equiv}
Two bilinear forms in~\eqref{bilinear_SSE1} and~\eqref{bilinear_SSE2} are identical, i.e., it satisfies that
\begin{equation*}
\intO \bD S_h \bepsilon [\bu] : S_h \bepsilon [\bv] \, d \Omega = \intO \bD P_{2,h} P_{1,h} \bepsilon [\bu] : P_{2,h} P_{1,h} \bepsilon [\bv] \, d \Omega, \quad \bu, \bv \in V_h.
\end{equation*}
\end{theorem}
\begin{proof}
For simplicity, we present the proof for the case of the 3-node triangular element only; the 4-node case can be proven by a similar argument.

Thanks to the polarization identity~\cite[Theorem~0.19]{Teschl:2009}, it suffices to show that
\begin{equation*}
\int_T \bD S_h \bepsilon[\bu]: S_h \bepsilon[\bu]\,d\Omega = \int_T \bD P_{2,h} P_{1,h} \bepsilon [\bu] : P_{2,h} P_{1,h} \bepsilon [\bu] \,d \Omega
\end{equation*}
for $\bu \in V_h$ and $T \in \cT_h$.
We take any $\bu \in V_h$ and write $\bepsilon = \bepsilon[\bu]$.
Assume for simplicity that $T$ is an interior element.
Let $T_i$, $i=1,2,3$ be the neighboring elements of $T$ in $\cT_h$; see Fig.~\ref{Fig:adjacent}(a).
We denote the values of $\bepsilon$ on elements $T$ and $T_i$ by $\bepsilon_{T}$ and $\bepsilon_{T_i}$, respectively.
As three-point Gaussian integration is exact for this case, we have
\begin{equation*}
\int_T \bD S_h \bepsilon: S_h \bepsilon \,d\Omega = \frac{|T|}{3} \sum_{i=1}^3  \bD (S_h \bepsilon ) (\mathrm{G}i) : (S_h \bepsilon ) (\mathrm{G}i), 
\end{equation*}
where the Gaussian points $\mathrm{G}1$, $\mathrm{G}2$, and $\mathrm{G}3$ are shown in Fig.~\ref{Fig:coord}(a).
By~\eqref{smoothed_intermediate} and~\eqref{Gaussian}, $( S_h \bepsilon ) (\mathrm{G}i)$ is computed as follows:
\begin{equation*} \begin{split}
( S_h \bepsilon ) (\mathrm{G}i) &= \frac{1}{2} ( \hat{\bepsilon}^{(i-1)} + \hat{\bepsilon}^{(i)}) \\
&= \frac{1}{2} \left[ \frac{1}{|T| + |T_{i-1}|} \left( \int_{T } \bepsilon \,d\Omega + \int_{T_{i-1}} \bepsilon \,d\Omega\right) + \frac{1}{|T| + |T_{i}|} \left( \int_{T} \bepsilon \,d\Omega + \int_{T_{i}} \bepsilon \,d\Omega \right) \right] \\
&= \frac{1}{2} \left( \frac{|T|\bepsilon_T + |T_{i-1}| \bepsilon_{T_{i-1}}}{|T| + |T_{i-1}|} + \frac{|T|\bepsilon_T + |T_{i}| \bepsilon_{T_{i}}}{|T| + |T_{i}|}\right),
\end{split} \end{equation*}
where the modulo 3 convention is used.
For the case when $T$ is an exterior element such that there is no neighboring element $T_i$ for some $i$, one may replace $T_i$ in the above equality by $T$ to obtain the corresponding result.

On the other hand, let $T_{1,i}$ and $T_{2,i}$, $i=1,2,3$ be the subregions in $\cT_{1,h}$ and $\cT_{2,h}$ that overlap with $T$, respectively; see Figs.~\ref{Fig:adjacent}(b) and~\ref{Fig:quad} for the 3-node and 4-node cases, respectively.
Since $P_{2,h} P_{1,h} \bepsilon$ is piecewise constant on $\cT_{2,h}$, we have
\begin{equation*} \begin{split}
\int_T \bD P_{2,h} P_{1,h} \bepsilon : P_{2,h} P_{1,h} \bepsilon \,d\Omega &= \sum_{i=1}^3 \int_{T_{2,i}} \bD P_{2,h} P_{1,h} \bepsilon : P_{2,h} P_{1,h} \bepsilon \,d \Omega \\
&= \frac{|T|}{3} \sum_{i=1}^3 \bD (P_{2,h}P_{1,h} \bepsilon)_{T_{2,i}} : (P_{2,h}P_{1,h} \bepsilon)_{T_{2,i}}, 
\end{split}\end{equation*}
where $(P_{2,h}P_{1,h} \bepsilon)_{T_{2,i}}$ denotes the value of $P_{2,h}P_{1,h} \bepsilon$ on $T_{2,i}$.
Because $P_{2,h}$ is a piecewise averaging operator, the value of $(P_{2,h}P_{1,h} \bepsilon)_{T_{2,i}}$ is the weighted average of $(P_{1,h} \bepsilon)_{T_{1, i-1}}$ and $(P_{1,h} \bepsilon)_{T_{1, i}}$ with their respective weights $|T_{1,i-1} \cap T_{2,i}|$ and $|T_{1,i} \cap T_{2,i}|$ with the modulo 3 convention.
Similarly, the value of $(P_{1,h} \bepsilon)_{T_{1, i}}$ is the weighted average of $\bepsilon_{T}$ and $\bepsilon_{T_i}$ with their respective weights $|T \cap T_{1,i}|$ and $|T_i \cap T_{1,i}|$.
That is, it follows that
\begin{equation*} 
\begin{split}
(P_{2,h}P_{1,h} \bepsilon)_{T_{2,i}}
&= \frac{1}{2} \left( (P_{1,h} \bepsilon)_{T_{1, i-1}} + (P_{1,h} \bepsilon)_{T_{1, i}}\right)
\quad \hspace{1.6cm} (\because \gap |T_{1,i-1} \cap T_{2,i}| = |T_{1,i} \cap T_{2,i}|)
\\
&= \frac{1}{2} \left( \frac{|T|\bepsilon_T + |T_{i-1}| \bepsilon_{T_{i-1}}}{|T| + |T_{i-1}|} + \frac{|T|\bepsilon_T + |T_{i}| \bepsilon_{T_{i}}}{|T| + |T_{i}|}\right),
\quad (\because \gap |T \cap T_{1,i}| : |T_i \cap T_{1,i}| = |T| : |T_i| )
\end{split} 
\end{equation*}
where $(P_{1,h} \bepsilon)_{T_{1,i}}$ is the value of  $P_{1,h} \bepsilon$ on $T_{1,i}$.
When $T$ is an exterior element, say $T_i$ does not exist for some $i$, we have
$$(P_{1,h} \bepsilon)_{T_{1,i}} = \bepsilon_T = \frac{|T|\bepsilon_T + |T| \bepsilon_{T}}{|T| + |T|}.$$
This completes the proof.
\end{proof}

As a direct consequence of Theorem~\ref{Thm:equiv}, two bilinear forms~\eqref{bilinear_SSE1} and~\eqref{bilinear_SSE2} provide the same displacement solution $\bar{\bu}_h \in V_h$ when they are adopted for~\eqref{weak_SSE}.
On the contrary, they have different distributions in smoothed strain fields;~\eqref{bilinear_SSE2} has piecewise constant fields within an element, whereas~\eqref{bilinear_SSE1} has a linear/bilinear field.
We close this section by presenting the uniqueness theorem for the solution of the SSE method.

\begin{proposition}
\label{Prop:unique}
The SSE method~\eqref{weak_SSE} has a unique solution.
\end{proposition}
\begin{proof}
The coercivity of the bilinear form $\bar{a}(\cdot, \cdot)$ in~\eqref{bilinear_SSE2} can be proven using the same argument as~\cite[Sect.~3.9]{Liu:2010}.
Subsequently, the uniqueness of a solution of~\eqref{weak_SSE} is straightforward by Theorem~\eqref{Thm:equiv} and the Lax--Milgram theorem~\cite[Theorem~2.7.7]{BS:2008}.
\end{proof}

\begin{remark}
Subdivisions $\cT_{1,h}$ and $\cT_{2,h}$ introduced in this section are not required in the implementation of the SSE method; they are for the sake of convergence analysis only.
In implementation, the SSE method only requires the original mesh~$\cT_h$~\cite{LL:2018}, while existing S-FEMs require additional subdivisions consisting of smoothing domains~\cite{LNN:2010,ZL:2018}.
\end{remark}

\section{A variational principle for the strain-smoothed element method}
\label{Sec:VP}
In this section, we construct a variational principle for linear elasticity with respect to a single displacement field, two strain fields, and two stress fields.
Subsequently, we demonstrate that the SSE method interpreted by the bilinear form~\eqref{bilinear_SSE2} is a Galerkin approximation of the constructed variational principle.
We note that, while S-FEM can be interpreted in terms of existing variational principles such as the Hellinger--Reissner and Hu--Washizu variational principles~(see~\cite{LNN:2010} and~~\cite{NBN:2008}, respectively), we are unable to derive the SSE method from those principles owing to the additional unification procedures introduced in Sect.~\ref{Sec:SSE}.
Throughout this section, let index $k$ denote either $1$ or $2$.

The starting point is the minimization problem~\eqref{weak_min}.
We set $W_k = W$.
Consider two independent strain fields $\bepsilon_1 \in W_1$ and $\bepsilon_2 \in W_2$.
It is obvious that~\eqref{weak_min} is equivalent to the following constrained minimization problem:
\begin{equation}
\label{epsilon_constrained}
\min_{\bu \in V, \textrm{ }\bepsilon_1 \in W_1, \textrm{ }\bepsilon_2 \in W_2} \left\{ \frac{1}{2}  \intO \bD \bepsilon_2 : \bepsilon_2 \,d\Omega - f (\bu) \right\}
\quad\textrm{subject to}\quad \bepsilon_1 = \bB \bu \textrm{ and } \bepsilon_1 = \bepsilon_2.
\end{equation}
In~\eqref{epsilon_constrained}, we use the method of Lagrange multipliers in order to deal with the constraints $\bepsilon_1 = \bB \bu$ and $\bepsilon_1 = \bepsilon_2$.
Then, we obtain the following saddle point problem:
\begin{multline}
\label{epsilon_saddle}
\min_{\bu \in V, \textrm{ }\bepsilon_1 \in W_1, \textrm{ }\bepsilon_2 \in W_2} \max_{\bsigma_1 \in W_1, \textrm{ }\bsigma_2 \in W_2} \Bigg\{ \frac{1}{2}  \intO \bD \bepsilon_2 : \bepsilon_2 \,d\Omega - f (\bu)
+ \intO \bsigma_1: (\bB \bu - \bepsilon_1) \,d\Omega + \intO \bsigma_2 : (\bepsilon_1 - \bepsilon_2) \,d\Omega \Bigg\},
\end{multline}
where $\bsigma_1 \in W_1$ and $\bsigma_2 \in W_2$ are the Lagrange multipliers corresponding to the constraints $\bepsilon_1 = \bB \bu$ and $\bepsilon_1 = \bepsilon_2$, respectively.
Equivalently, we have the following variational problem: find $(\bu, \bepsilon_1, \bepsilon_2, \bsigma_1, \bsigma_2 ) \in V \times W_1 \times W_2 \times W_1 \times W_2$ such that
\begin{equation}
\begin{split}
\label{VP}
\intO \bsigma_1 : \bB \bv \,d\Omega + \intO (-\bsigma_1 + \bsigma_2 ) : \bdelta_1 \,d\Omega + \intO ( \bD \bepsilon_2 - \bsigma_2) : \bdelta_2 \,d\Omega = f (\bv) \quad \forall &\bv \in V, \gap \bdelta_1 \in W_1, \gap \bdelta_2 \in W_2, \\
\intO \btau_1 : (\bB \bu - \bepsilon_1 ) \,d\Omega + \intO \btau_2 : (\bepsilon_1 - \bepsilon_2) \,d\Omega = 0 \quad \forall &\btau_1 \in W_1, \gap \btau_2 \in W_2.
\end{split}
\end{equation}
The existence and uniqueness of a solution of the variational principle~\eqref{VP} are summarized in Proposition~\ref{Prop:VP}.
We postpone the proof of Proposition~\ref{Prop:VP} until Sect.~\ref{Sec:Convergence}.
A more general statement will be given in Proposition~\ref{Prop:abstract_VP}.

\begin{proposition}
\label{Prop:VP}
The variational problem~\eqref{VP} has a unique solution $(\bu, \bepsilon_1, \bepsilon_2, \bsigma_1, \bsigma_2 ) \in V \times W_1 \times W_2 \times W_1 \times W_2$.
Moreover, $\bu$ solves~\eqref{weak} and the following relations hold:
\begin{equation*}
\bepsilon_1 = \bepsilon_2 = \bB \bu, \quad \bsigma_1 = \bsigma_2 = \bD \bB \bu.
\end{equation*}
\end{proposition}

\begin{remark}
\label{Rem:Lagrange}
Proposition~\ref{Prop:VP} shows that the Lagrange multipliers $\bsigma_1$ and $\bsigma_2$ introduced in~\eqref{epsilon_saddle} play a role of the stress field.
\end{remark}

\begin{remark}
\label{Rem:HW}
The elimination of two variables $\bepsilon_2$ and $\bsigma_2$ in~\eqref{epsilon_saddle} yields
\begin{equation*}
\min_{\bu \in V, \textrm{ }\bepsilon_1 \in W_1} \max_{\bsigma_1 \in W_1} \Bigg\{ \frac{1}{2}  \intO \bD \bepsilon_1 : \bepsilon_1 \,d\Omega - f (\bu)
+ \intO \bsigma_1: (\bB \bu - \bepsilon_1) \,d\Omega \Bigg\},
\end{equation*}
which is the Hu--Washizu variational principle.
In this sense, we can say that~\eqref{VP} generalizes the Hu--Washizu variational principle.
\end{remark}

\subsection{Galerkin approximation}
We consider a Galerkin approximation of~\eqref{VP} made by replacing the spaces $V$ and $W_k$ by their finite-dimensional subspaces $V_h \subset V$ and $W_{k,h} \subset W_k$, respectively~(see Sect.~\ref{Sec:SSE} for the definitions of $V_h$ and $W_{k,h}$): find $(\bar{\bu}_h, \bepsilon_{1,h}, \bepsilon_{2,h}, \bsigma_{1,h}, \bsigma_{2,h} ) \in V_h \times W_{1,h} \times W_{2,h} \times W_{1,h} \times W_{2,h}$ such that
\begin{subequations}
\label{Galerkin}
\begin{footnotesize}
\begin{align}
\label{Galerkin1}
\intO \bsigma_{1,h} : \bB \bv \,d\Omega + \intO (-\bsigma_{1,h} + \bsigma_{2,h} ) : \bdelta_1 \,d\Omega + \intO ( \bD \bepsilon_{2,h} - \bsigma_{2,h}) : \bdelta_2 \,d\Omega = f (\bv) \quad \forall &\bv \in V_h, \gap \bdelta_1 \in W_{1,h}, \gap \bdelta_2 \in W_{2,h},\\
\label{Galerkin2}
\intO \btau_1 : (\bB \bar{\bu}_h - \bepsilon_{1,h} ) \,d\Omega + \intO \btau_2 : (\bepsilon_{1,h} - \bepsilon_{2,h}) \,d\Omega = 0 \quad \forall &\btau_1 \in W_{1,h}, \gap \btau_2 \in W_{2,h}.
\end{align}
\end{footnotesize}
\end{subequations}
We take $\bv = \mathbf{0}$ and $\bdelta_2 = \mathbf{0}$ in~\eqref{Galerkin1}.
Then we have
\begin{equation*}
\intO (-\bsigma_{1,h} + \bsigma_{2,h} ) : \bdelta_1 \,d\Omega = 0 \quad \forall \bdelta_1 \in W_{1,h},
\end{equation*}
which implies that $\bsigma_{1,h}$ is the $(L^2 (\Omega))^3$-orthogonal projection of $\bsigma_{2,h}$ onto $W_{1,h}$.
It follows by Lemma~\ref{Lem:orthogonal} that
\begin{equation*}
\bsigma_{1,h} = P_{1,h} \bsigma_{2,h}.
\end{equation*}
Similarly, it is straightforward to verify that
\begin{equation*}
\bsigma_{2,h} = \bD \bepsilon_{2,h}
\end{equation*}
from~\eqref{Galerkin1} and that
\begin{equation*}
\bepsilon_{1,h} = P_{1,h} ( \bB \bar{\bu}_h ), \quad
\bepsilon_{2,h} = P_{2,h} \bepsilon_{1,h}
\end{equation*}
from~\eqref{Galerkin2}.
Using the above relations and Lemmas~\ref{Lem:commute} and~\ref{Lem:orthogonal}, we readily get
\begin{equation*} 
\bsigma_{1,h} = P_{1,h} \bD P_{2,h} P_{1,h} (\bB \bar{\bu}_h ) = P_{1,h} P_{2,h} \left( \bD P_{2,h} P_{1,h} (\bB \bar{\bu}_h) \right).
\end{equation*}
Substituting $\bdelta_1 = \mathbf{0}$ and $\bdelta_2 = \mathbf{0}$ in~\eqref{Galerkin1} yields
\begin{equation*}
\intO \bD P_{2,h} P_{1,h} (\bB\bar{\bu}_h ) : P_{2,h} P_{1,h} ( \bB \bv ) \,d\Omega = f (\bv) \quad \forall \bv \in V_h,
\end{equation*}
that is equivalent to~\eqref{weak_SSE} with the bilinear form $\bar{a}(\cdot, \cdot)$ given in~\eqref{bilinear_SSE2}.
Therefore, the SSE method can be derived from the variational principle~\eqref{VP}.
We summarize the above discussion in the following theorem.
Note that the uniqueness of the solution of the SSE method was presented in Proposition~\ref{Prop:unique}.

\begin{theorem}
\label{Thm:VP}
The variational problem~\eqref{Galerkin} has a unique solution $(\bar{\bu}_h, \bepsilon_{1,h}, \bepsilon_{2,h}, \bsigma_{1,h}, \bsigma_{2,h} ) \in V_h \times W_{1,h} \times W_{2,h} \times W_{1,h} \times W_{2,h}$ that satisfies
\begin{equation*}
\bepsilon_{1,h} = P_{1,h}(\bB \bar{\bu}_h), \quad
\bepsilon_{2,h} = P_{2,h} P_{1,h}(\bB \bar{\bu}_h), \quad
\bsigma_{1,h} = P_{1,h} P_{2,h} \left( \bD P_{2,h} P_{1,h} (\bB \bar{\bu}_h) \right), \quad
\bsigma_{2,h} = \bD P_{2,h} P_{1,h}( \bB \bar{\bu}_h),
\end{equation*}
and $\bar{\bu}_h$ is a unique solution of~\eqref{weak_SSE} with the bilinear form $\bar{a}(\cdot, \cdot)$ given in~\eqref{bilinear_SSE2}.
\end{theorem}

\section{Convergence analysis}
\label{Sec:Convergence}
In this section, we present a convergence theory for the SSE method based on the variational formulation~\eqref{VP}.
To present a unified convergence analysis for the standard FEM, S-FEM, and SSE method, the convergence theory established in this section is built upon an abstract mixed problem that generalizes~\eqref{VP}.

Let $X$ and $Y$ be two Hilbert spaces equipped with inner products $\left< \cdot, \cdot \right>_X$ and $\left< \cdot, \cdot \right>_Y$ and their induced norms $\| \cdot \|_X$ and $\| \cdot \|_Y$, respectively.
We set $\Pi = X \times Y \times Y$ and $\Delta = Y \times Y$.
Let $D$:~$Y \rightarrow Y$ be a continuous and symmetric positive definite linear operator, so that
\begin{equation*}
\vertiii{ \epsilon }_{Y} = \langle D \epsilon, \epsilon \rangle_Y^{1/2}, \quad \epsilon \in Y
\end{equation*}
becomes a norm on $Y$.
In this case, the dual norm $\vertiii{\cdot}_{Y^*}$ of $\vertiii{\cdot}_Y$ is given as follows:
\begin{equation*}
\vertiii{\sigma}_{Y^*} = \sup_{\delta \in Y \setminus \{0\}} \frac{\left< \sigma, \delta \right>_Y}{\vertiii{\delta}_Y} = \langle \sigma, D^{-1} \sigma \rangle_Y^{1/2}, \quad \sigma \in Y.
\end{equation*}
We additionally assume that there is a continuous linear operator $B$:~$X \rightarrow Y$ such that
\begin{equation*}
\vertiii{u}_X = \vertiii{Bu}_{Y}, \quad u \in X
\end{equation*}
becomes a norm on $X$.
The following norms on the spaces $\Pi$ and $\Delta$ are defined:
\begin{equation*} \begin{split}
\vertiii{U}_{\Pi}^2 &= \vertiii{u}_X^2 + \vertiii{\epsilon_1}_Y^2 + \vertiii{\epsilon_2}_Y^2, \quad U = (u, \epsilon_1, \epsilon_2 ) \in \Pi, \\
\vertiii{P}_{\Delta}^2 &= \vertiii{\epsilon_1}_{Y}^2 + \vertiii{\epsilon_2}_{Y}^2, \quad \hspace{1.05cm} P = (\epsilon_1, \epsilon_2) \in \Delta, \\
\vertiii{Q}_{\Delta^*}^2 &= \vertiii{\sigma_1}_{Y^*}^2 + \vertiii{\sigma_2}_{Y^*}^2, \quad \hspace{0.6cm} Q = (\sigma_1, \sigma_2) \in \Delta.
\end{split} \end{equation*}
We also define a seminorm $| \cdot |_{\Pi}$ on $\Pi$ as follows:
\begin{equation*}
|U|_{\Pi} = \vertiii{\epsilon_2}_Y, \quad U = (u, \epsilon_1, \epsilon_2 ) \in \Pi.
\end{equation*}

Let $\cD$:~$\Pi \rightarrow \Delta$ be a linear operator given by
\begin{equation}
\label{cD}
\cD U = (Bu - \epsilon_1, \epsilon_1 - \epsilon_2), \quad U=(u, \epsilon_1, \epsilon_2) \in \Pi.
\end{equation}
In terms of the operator $\cD$, we define the bilinear form $B(\cdot, \cdot)$:~$\Pi \times \Delta \rightarrow \mathbb{R}$ as follows:
\begin{equation*}
B(V, Q) = \left< \cD V, Q \right>_{\Delta} = \left< \tau_1 , Bv - \delta_1 \right>_Y + \left< \tau_2 , \delta_1 - \delta_2 \right>_Y, \quad V=(v, \delta_1, \delta_2 ) \in \Pi, \gap Q=(\tau_1, \tau_2) \in \Delta.
\end{equation*}
It is straightforward to verify that the kernel $Z$ of $B(\cdot, \Delta)$ defined by
\begin{equation}
\label{Z}
Z = \left\{ V \in \Pi :\gap B(V, Q) = 0, \gap Q \in \Delta \right\}
\end{equation}
is characterized as follows:
\begin{equation}
\label{Z_char}
Z = \left\{ (v, Bv, Bv) \in \Pi :\gap v \in X \right\}.
\end{equation}
The seminorm $| \cdot |_{\Pi}$ is positive definite on $Z$ since
\begin{equation}
\label{Z_positive}
|U|_{\Pi}^2 = \vertiii{Bu}_Y^2 = \frac{1}{3} \vertiii{U}_{\Pi}^2, \quad U =( u, Bu, Bu) \in Z.
\end{equation}
In other words, $| \cdot |_{\Pi}$ becomes a norm on $Z$.

If we define a bilinear form $A(\cdot, \cdot)$:~$\Pi \times \Pi \rightarrow \mathbb{R}$ by
\begin{equation*}
A(U,V) = \langle D \epsilon_2, \delta_2 \rangle_Y, \quad U=(u,\epsilon_1,\epsilon_2), \gap V=(v, \delta_1, \delta_2) \in \Pi,
\end{equation*}
then it is continuous and coercive with respect to $| \cdot |_{\Pi}$ since
\begin{equation} \begin{split}
\label{A_continuous}
A(U,V) = \langle D \epsilon_2, \delta_2 \rangle_Y
\leq \vertiii{\epsilon_2}_Y \vertiii{\delta_2}_Y 
= | U |_{\Pi} | V |_{\Pi}
\end{split} \end{equation}
and
\begin{equation} \begin{split}
\label{A_coercive}
A(U, U) = \vertiii{\epsilon_2}_Y^2 = |U|_{\Pi}^2 
\end{split} \end{equation}
for any $U=(u,\epsilon_1, \epsilon_2),V=(v, \delta_1, \delta_2) \in \Pi$.

We are now ready to state the following abstract variational problem to find $U \in \Pi$ and $P \in \Delta$ such that
\begin{subequations}
\label{abstract_VP}
\begin{align}
\label{abstract_VP1}
A(U,V) + B(V,P) &= F(V) \quad \forall V \in \Pi, \\
\label{abstract_VP2}
B(U,Q) &= 0 \hspace{0.6cm} \quad \forall Q \in \Delta,
\end{align}
\end{subequations}
where $F \in \Pi^*$ satisfies
\begin{equation*}
F(V) = f(v), \quad V = (v, \delta_1, \delta_2) \in \Pi,
\end{equation*}
for some $f \in X^*$.
The existence and uniqueness of a solution of~\eqref{abstract_VP} can be shown as follows.

\begin{proposition}
\label{Prop:abstract_VP}
The variational problem~\eqref{abstract_VP} has a unique solution $(U, P) \in \Pi \times \Delta$.
Moreover, the unique solution $(U, P)$ is characterized by
\begin{equation*}
U = (u, Bu, Bu), \quad P = (DBu, DBu),
\end{equation*}
where $u \in X$ is a unique solution of the variational problem
\begin{equation}
\label{reduced_VP}
\langle D B u, B v \rangle_Y = f(v) \quad \forall v \in X.
\end{equation}
\end{proposition}
\begin{proof}
The existence and uniqueness of a solution of~\eqref{reduced_VP} are direct consequences of the Lax--Milgram theorem~\cite[Theorem~2.7.7]{BS:2008}.
The equation~\eqref{abstract_VP2} implies that $U \in Z$.
By~\eqref{abstract_VP1}, $U$ can be determined by the following variational problem: find $U \in Z$ such that
\begin{equation}
\label{Z_VP}
A(U,V) = F(V) \quad \forall V \in Z.
\end{equation}
Because $|\cdot|_{\Pi}$ is a norm on $Z$~(see~\eqref{Z_positive}), the existence and uniqueness of $U$ are guaranteed by~\eqref{A_continuous},~\eqref{A_coercive}, and the Lax--Milgram theorem applied to~\eqref{Z_VP}.
By~\eqref{Z_char}, we have $U=(u, Bu, Bu)$ for some $u \in X$.
Writing $V = (v, Bv, Bv)$ for $v \in X$, the problem~\eqref{Z_VP} is reduced to~\eqref{reduced_VP}.
Therefore, $u$ is a unique solution of~\eqref{reduced_VP}.

Next, we characterize the dual solution $P$.
We write $V = (v, \delta_1, \delta_2 )$ and $P = (\sigma_1, \sigma_2)$ in~\eqref{abstract_VP1}.
Substituting $v=0$ and $\delta_2 = 0$ in~\eqref{abstract_VP1} yields
\begin{equation*}
\left< \sigma_1 - \sigma_2 , \delta_1 \right>_Y = 0 \quad \forall \delta_1 \in Y,
\end{equation*}
which is equivalent to $\sigma_1 = \sigma_2$.
Meanwhile, by substituting $U=(u, Bu, Bu)$, $v=0$, and $\delta_1 = 0$ in~\eqref{abstract_VP1}, we have
\begin{equation*}
\left< D B u - \sigma_2 , \delta_2 \right>_Y = 0 \quad \forall \delta_2 \in Y.
\end{equation*}
That is, we get $\sigma_2 = DBu$.
Therefore, we conclude that $\sigma_1 = \sigma_2 = DBu$.
\end{proof}

The abstract problem~\eqref{abstract_VP} generalizes several important elliptic partial differential equations.
If we set
\begin{equation*}
X = \left\{ u \in H^1 (\Omega) : u=0 \textrm{ on } \Gamma_D \right\}, \quad Y = L^2 (\Omega), \quad D = I, \quad B = \nabla
\end{equation*}
in~\eqref{abstract_VP}, then~\eqref{reduced_VP} becomes
\begin{equation*}
\intO \nabla u \cdot \nabla v \,d\Omega = f(v) \quad \forall v \in X,
\end{equation*}
which is the weak formulation for the Poisson's equation with a mixed boundary condition~\cite{BS:2008,Ciarlet:2002}.
Meanwhile, if we set
\begin{equation}
\label{abstract_elasticity}
X = V, \quad Y = W, \quad D = \bD, \quad B = \bB,
\end{equation}
where $V$, $W$, $\bD$, and $\bB$ were defined in Sect.~\ref{Sec:Elasticity}, then~\eqref{abstract_VP} and~\eqref{reduced_VP} are reduced to~\eqref{VP} and~\eqref{weak}, respectively.
Therefore, linear elasticity is an instance of~\eqref{abstract_VP}.
In this sense, Proposition~\ref{Prop:abstract_VP} generalizes Proposition~\ref{Prop:VP}.

Now, we present a Galerkin approximation of~\eqref{abstract_VP} which generalizes~\eqref{Galerkin}.
Let $X_h \subset X$, $Y_{1,h} \subset Y$, and $Y_{2,h} \subset Y$.
For $\Pi_h = X_h \times Y_{1,h} \times Y_{2,h}$ and $\Delta_h = Y_{1,h} \times Y_{2,h}$, we consider a variational problem to find $U_h \in \Pi_h$ and $P_h \in \Delta_h$ such that
\begin{equation}
\label{abstract_Galerkin}
\begin{split}
A(U_h,V) + B(V,P_h) &= F(V) \quad \forall V \in \Pi_h, \\
B(U_h,Q) &= 0 \hspace{0.6cm} \quad \forall Q \in \Delta_h.
\end{split} \end{equation}
Similarly to~\eqref{Z}, we define
\begin{equation}
\label{Z_h}
Z_h = \left\{ V \in \Pi_h :\gap B(V, Q) = 0, \gap Q \in \Delta_h \right\}.
\end{equation}
Note that $Z_h \not\subset Z$ in general.
We state an assumption on $Z_h$ that is necessary to obtain a bound for the error $U - U_h$.

\begin{assumption}
\label{Ass:positive}
The seminorm $|\cdot|_{\Pi}$ is positive definite on $Z \cup Z_h$, i.e., there exists a positive constant $\alpha$ such that
\begin{equation*}
|U|_{\Pi} \geq \alpha \vertiii{U}_{\Pi}, \quad U \in Z \cup Z_h.
\end{equation*}
\end{assumption}

Thanks to~\eqref{Z_positive}, it is enough to prove the positive definiteness of $| \cdot |_{\Pi}$ on $Z_h$ in order to verify Assumption~\ref{Ass:positive} in applications.
Under Assumption~\ref{Ass:positive}, the primal solution $U_h$ of~\eqref{abstract_Galerkin} is uniquely determined since it solves
\begin{equation}
\label{Z_Galerkin}
A(U_h,V) = F(V) \quad \forall V \in Z_h.
\end{equation}
Moreover, one can prove the following continuity condition of the bilinear form $B(\cdot, \cdot)$ with respect to $|\cdot|_{\Pi}$.

\begin{lemma}
\label{Lem:B_continuous}
Suppose that Assumption~\ref{Ass:positive} holds.
Then, there exists a positive constant $C_B$ such that
\begin{equation*}
B(V,Q) \leq C_B |V|_{\Pi} \vertiii{Q}_{\Delta^*}, \quad V \in \Pi, \gap P \in \Delta.
\end{equation*}
\end{lemma}
\begin{proof}
First, we show that the operator $\cD$ defined in~\eqref{cD} is bounded.
For any $U=(u, \epsilon_1, \epsilon_2) \in \Pi$, it follows that
\begin{equation} \begin{split}
\label{C_D}
\vertiii{\cD U}_{\Delta}^2 &= \vertiii{Bu - \epsilon_1}_Y^2 + \vertiii{\epsilon_1 - \epsilon_2}_Y^2 \\
&\leq 2 ( \vertiii{Bu}_Y^2 + \vertiii{\epsilon_1}_Y^2 ) + 2 ( \vertiii{\epsilon_1}_Y^2 + \vertiii{\epsilon_2}_Y^2 ) \\
&= 2 \vertiii{u}_X^2 + 4 \vertiii{\epsilon_1}_Y^2 + 2 \vertiii{\epsilon_2}_Y^2 \\
&\leq 4 \vertiii{U}_{\Pi}^2.
\end{split} \end{equation}
Using~\eqref{C_D}, one can obtain the desired result with $C_B = 2/\alpha$ as follows: for $V \in \Pi$ and $Q \in \Delta$, we have
\begin{equation*} \begin{split}
B(V, Q) &= \left< \cD V, Q \right>_{\Delta} \\
&\leq \vertiii{\cD V}_{\Delta} \vertiii{Q}_{\Delta^*} \\
&\stackrel{\eqref{C_D}}{\leq} 2 \vertiii{V}_{\Pi} \vertiii{Q}_{\Delta^*} \\
&\leq \frac{2}{\alpha} |V|_{\Pi} \vertiii{Q}_{\Delta^*},
\end{split} \end{equation*}
where we used Assumption~\ref{Ass:positive} in the last inequality.
\end{proof}

Motivated by~\cite[Theorem~12.3.7]{BS:2008}, we have the following result on the relation between the primal solutions of the variational problem~\eqref{abstract_VP} and its Galerkin approximation~\eqref{abstract_Galerkin}.

\begin{theorem}
\label{Thm:abstract_Galerkin}
Suppose that Assumption~\ref{Ass:positive} holds.
Let $(U, P) \in \Pi \times  \Delta$ be a unique solution of~\eqref{abstract_VP}, and let $U_h \in \Pi_h$ be a unique primal solution of~\eqref{abstract_Galerkin}.
Then we have
\begin{equation*}
| U - U_h |_{\Pi} \leq 2 \inf_{V \in Z_h} | U - V |_{\Pi} + C_B \inf_{Q \in \Delta_h} \vertiii{P-Q}_{\Delta^*},
\end{equation*}
where $C_B$ was defined in Lemma~\ref{Lem:B_continuous}.
\end{theorem}
\begin{proof}
Note that $U$ and $U_h$ solve~\eqref{Z_VP} and~\eqref{Z_Galerkin}, respectively.
Thanks to~\eqref{A_continuous},~\eqref{A_coercive}, and Assumption~\ref{Ass:positive}, one can apply Theorem~\ref{Thm:nonconforming} to obtain
\begin{equation}
\label{ingredient1}
|U - U_h|_{\Pi} \leq 2 \inf_{V \in Z_h} |U-V|_{\Pi} + \sup_{W \in Z_h \setminus \{0 \}} \frac{| A(U-U_h, W) |}{|W|_{\Pi}}.
\end{equation}
Meanwhile, for any $W \in Z_h$ and $Q \in \Delta_h$, we have
\begin{equation} \begin{split}
\label{ingredient2}
| A(U-U_h, W) | &\stackrel{\eqref{Z_Galerkin}}{=} | A(U, W) - F(W)| \\
&\stackrel{\eqref{abstract_VP1}}{=} | B(W, P) | \\
&\stackrel{\eqref{Z_h}}{=} | B(W, P-Q) | \\
&\leq C_B | W |_{\Pi} \vertiii{P-Q}_{\Delta^*},
\end{split}\end{equation}
where the last inequality is due to Lemma~\ref{Lem:B_continuous}.
Combining~\eqref{ingredient1} and~\eqref{ingredient2} yields the desired result.
\end{proof}

As linear elasticity is an instance of the continuous problem~\eqref{abstract_VP}, various FEMs such as the standard FEM, S-FEM, and SSE method for linear elasticity can be written in the form of~\eqref{abstract_Galerkin}.
We present how the convergence results of these methods can be obtained in a unified manner from Theorem~\ref{Thm:abstract_Galerkin}.
In what follows, we assume the setting~\eqref{abstract_elasticity}.
Subsequently, the norms $\vertiii{\cdot}_Y$ and $\vertiii{\cdot}_{Y^*}$ become the energy norms for the strain and stress fields, respectively, i.e.,
\begin{equation*}
\vertiii{ \bepsilon }_Y^2 = \intO \bD \bepsilon : \bepsilon \, d\Omega, \quad \bepsilon \in W,
\end{equation*}
and
\begin{equation*}
\vertiii{ \bsigma }_{Y^*}^2 = \intO \bsigma : \bD^{-1} \bsigma \, d\Omega, \quad \bsigma \in W.
\end{equation*}

\subsection{Standard finite element method}
\label{Subsec:StdFEM}
First, we set $X_h = V_h$ and $Y_{1,h} = Y_{2,h} = W_h$ in~\eqref{abstract_Galerkin}, where the spaces $V_h$ and $W_h$ were defined in Sect.~\ref{Sec:SSE}.
Since the meshes associated with $V_h$ and $W_h$ agree, it satisfies $\bB \bv \in  W_h$ for all $\bv \in V_h$.
Accordingly, the set~$Z_h$ defined in~\eqref{Z_h} is characterized by
\begin{equation*}
Z_h = \left\{ (\bv, \bB\bv, \bB\bv) \in V_h \times W_h \times W_h : \bv \in V_h \right\}.
\end{equation*}
In addition, the variational problem~\eqref{Z_Galerkin} reduces to the standard FEM formulation
\begin{equation}
\label{eg_FEM}
\intO \bD \bepsilon[\bu_h] : \bepsilon[\bv] \,d\Omega = f (\bv) \quad \forall \bv \in V_h,
\end{equation}
where $\bepsilon [\bv] = \bB \bv$.

For $\bV = (\bv, \bB \bv, \bB \bv ) \in Z_h$, one can easily verify that
\begin{equation*}
\vertiii{ \bV }_{\Pi}^2 = 3  \vertiii{ \bepsilon[\bv] }_{Y}^2 = 3 |\bV|_{\Pi}^2,
\end{equation*}
which implies that Assumption~\ref{Ass:positive} holds.
Therefore, one can obtain an error estimate for~\eqref{eg_FEM} as a simple corollary of Theorem~\ref{Thm:abstract_Galerkin} as follows.

\begin{corollary}
\label{Cor:FEM}
Let $\bu \in V$ and $\bu_h \in V_h$ solve~\eqref{weak} and~\eqref{eg_FEM}, respectively.
Then we have
\begin{equation*}
\vertiii{ \bepsilon [\bu] - \bepsilon [\bu_h] }_{Y} \leq 2 \inf_{\bv \in V_h} \vertiii{ \bepsilon [\bu] - \bepsilon [\bv] }_{Y} + C_B \left( \inf_{\btau_1 \in W_h} \vertiii{ \bsigma[\bu] - \btau_1 }_{Y^*} + \inf_{\btau_2 \in W_h} \vertiii{ \bsigma[\bu] - \btau_2 }_{Y^*}\right),
\end{equation*}
where
\begin{equation*}
\bepsilon[\bv] = \bB \bv, \gap \bsigma[\bv] = \bD \bB \bv, \quad \bv \in V_h,
\end{equation*}
and $C_B$ was defined in Assumption~\ref{Ass:positive}.
\end{corollary}

\subsection{Edge-based smoothed finite element method}
Next, let $X_h = V_h$ and $Y_{1,h} = Y_{2,h} = W_{1,h}$ in~\eqref{abstract_Galerkin}, where the space $W_{1,h}$ was defined in Sect.~\ref{Sec:SSE}.1.
By a similar argument as Sect.~\ref{Sec:VP}.1, we get
\begin{equation*}
Z_h = \left\{ (\bv, P_{1,h} ( \bB \bv), P_{1,h} ( \bB \bv)) \in V_h \times W_{1,h} \times W_{1,h} :  \bv \in V_h \right\}.
\end{equation*}
In this case, the variational problem~\eqref{Z_Galerkin} becomes the following: find $\hat{\bu}_h \in V_h$ such that
\begin{equation}
\label{eg_S_FEM}
\intO \bD \hat{\bepsilon}[\hat{\bu}_h] : \hat{\bepsilon}[\bv] \,d\Omega = f (\bv) \quad \forall \bv \in V_h,
\end{equation}
where $\hat{\bepsilon}[\bv] = P_{1,h}(\bB \bv)$.
It was shown in~\cite{LNN:2010} that~\eqref{eg_S_FEM} is a formulation for the edge-based S-FEM~\cite{LNL:2009}.

In order to verify Assumption~\ref{Ass:positive} for~\eqref{eg_S_FEM}, we first observe that
\begin{equation*}
\vertiii{\bV}_{\Pi}^2 = \vertiii{ \bepsilon [\bv] }_{Y}^2 + 2 \vertiii{ \hat{\bepsilon} [\bv] }_{Y}^2, \quad |\bV|_{\Pi}^2 = \vertiii{ \hat{\bepsilon}[\bv] }_{Y}^2
\end{equation*}
for $\bV = (\bv, P_{1,h} ( \bB \bv), P_{1,h} ( \bB \bv)) \in Z_h$.
Since it was shown in~\cite[Sect.~3.9]{Liu:2010} that there exists a positive constant $C$ such that
\begin{equation*}
\vertiii{ \hat{\bepsilon} [\bv] }_{Y} \geq C \vertiii{ \bepsilon [\bv] }_{Y}, \quad \bv \in V_h,
\end{equation*}
it is clear that Assumption~\ref{Ass:positive} holds.
The following corollary summarizes the convergence property of~\eqref{eg_S_FEM}~(cf.~\cite[Theorem~1]{LNN:2010}).

\begin{corollary}
\label{Cor:S_FEM}
Let $\bu \in V$ and $\hat{\bu}_h \in V_h$ solve~\eqref{weak} and~\eqref{eg_S_FEM}, respectively.
Then we have
\begin{equation*}
\vertiii{ \bepsilon [\bu] - \hat{\bepsilon} [\hat{\bu}_h] }_{Y} \leq 2 \inf_{\bv \in V_h} \vertiii{ \bepsilon [\bu] - \hat{\bepsilon} [\bv] }_{Y} + C_B \left( \inf_{\btau_1 \in W_{1,h}} \vertiii{ \bsigma[\bu] - \btau_1 }_{Y^*} + \inf_{\btau_2 \in W_{1,h}} \vertiii{ \bsigma[\bu] - \btau_2 }_{Y^*} \right),
\end{equation*}
where 
\begin{equation*}
\bepsilon[\bv] = \bB \bv, \gap \bsigma[\bv] = \bD \bB \bv, \gap \hat{\bepsilon}[\bv] = P_{1,h}(\bB \bv), \quad \bv \in V_h,
\end{equation*}
and $C_B$ was defined in Assumption~\ref{Ass:positive}.
\end{corollary}

\subsection{Strain-smoothed element method}
In order to derive the formulation for the SSE method~\eqref{weak_SSE} from the abstract problem~\eqref{abstract_Galerkin}, we set $X_h = V_h$, $Y_{1,h} = W_{1,h}$, and $Y_{2,h} = W_{2,h}$, where the space $W_{2,h}$ was defined in Sect.~\ref{Sec:SSE}.1.
Then the set $Z_h$ is characterized by
\begin{equation*}
Z_h = \left\{ (\bv, P_{1,h} (\bB \bv), P_{2,h} P_{1,h} (\bB \bv)) \in V_h \times W_{1,h} \times W_{2,h} :  \bv \in V_h \right\},
\end{equation*}
and~\eqref{Z_Galerkin} is reduced to~\eqref{weak_SSE}: find $\bar{\bu}_h \in V_h$ such that
\begin{equation}
\label{eg_SSE}
\intO \bD \bar{\bepsilon}[\bar{\bu}_h] : \bar{\bepsilon}[\bv] \,d\Omega = f (\bv) \quad \forall \bv \in V_h,
\end{equation}
where $\bar{\bepsilon}[\bv] = P_{2,h}P_{1,h} (\bB \bv)$.

Similar to the case of S-FEM, we have
\begin{equation*}
\vertiii{\bV}_{\Pi}^2 = \vertiii{ \bepsilon [\bv] }_{Y}^2 +  \vertiii{ \hat{\bepsilon} [\bv] }_{Y}^2 + \vertiii{ \bar{\bepsilon}[\bv] }_{Y}^2, \quad
|\bV|_{\Pi}^2 = \vertiii{ \bar{\bepsilon}[\bv] }_{Y}^2
\end{equation*}
for $\bV = (\bv, P_{1,h}(\bB \bv), P_{2,h}P_{1,h} (\bB \bv)) \in Z_h$.
Using the same argument as in~\cite[Sect.~3.9]{Liu:2010}, one can show without major difficulty that there exists a positive constant $C$ such that
\begin{equation*}
\vertiii{ \bar{\bepsilon} [\bv] }_{Y} \geq C \vertiii{ \hat{\bepsilon} [\bv] }_{Y}, \quad \bv \in V_h.
\end{equation*}
Hence, Assumption~\ref{Ass:positive} holds for~\eqref{eg_SSE}.
Finally, we have the following convergence theorem for the SSE method.

\begin{corollary}
\label{Cor:SSE}
Let $\bu \in V$ and $\bar{\bu}_h \in V_h$ solve~\eqref{weak} and~\eqref{eg_SSE}, respectively.
Then, we have
\begin{equation*}
\vertiii{ \bepsilon [\bu] - \bar{\bepsilon} [\bar{\bu}_h] }_{Y} \leq 2 \inf_{\bv \in V_h} \vertiii{ \bepsilon [\bu] - \bar{\bepsilon} [\bv] }_{Y} + C_B \left( \inf_{\btau_1 \in W_{1,h}} \vertiii{ \bsigma[\bu] - \btau_1 }_{Y^*} + \inf_{\btau_2 \in W_{2,h}} \vertiii{ \bsigma[\bu] - \btau_2 }_{Y^*}\right),
\end{equation*}
where 
\begin{equation*}
\bepsilon[\bv] = \bB \bv, \gap \bsigma[\bv] = \bD \bB \bv, \gap \hat{\bepsilon}[\bv] = P_{1,h}(\bB \bv), \gap \bar{\bepsilon}[\bv] = P_{2,h}P_{1,h}(\bB \bv), \quad \bv \in V_h,
\end{equation*}
and $C_B$ was defined in Assumption~\ref{Ass:positive}.
\end{corollary}

A conventional explanation for the fast convergence of strain smoothing methods is that the strain smoothing procedure allows the element to have more supporting nodes than the element so that the smoothed strain can be constructed by utilizing information in a broader region.
Meanwhile, the above convergence theorems allow us to develop a more quantitative explanation of why the SSE method converges faster than the standard FEM and edge-based S-FEM.
As discussed above, all these methods are conforming Galerkin approximations of the proposed variational principle~\eqref{VP} but use different finite-dimensional subspaces for strain approximation.
More precisely, the standard FEM, edge-based S-FEM, and SSE method use $(W_h, W_h)$, $(W_{1,h}, W_{1,h})$, and $(W_{1,h}, W_{2,h})$ as finite-dimensional approximations for $W \times W$, respectively.
Corollaries~\ref{Cor:FEM},~\ref{Cor:S_FEM}, and~\ref{Cor:SSE} indicate that the convergence rates of these methods depend on the approximabilities of the discrete spaces $W_h$, $W_{1,h}$, and $W_{2,h}$ defined on the subdivisions $\cT_h$, $\cT_{1,h}$, and $\cT_{2,h}$, respectively, for the continuous space $W$.
As depicted in Fig.~\ref{Fig:subdivision}, $\cT_{2,h}$ is a refinement of $\cT_h$, so that $W_h \subset W_{2,h}$.
Hence, we have
\begin{equation*}
\inf_{\btau \in W_{2,h}} \vertiii{ \bsigma - \btau }_{Y^*} \leq \inf_{\btau \in W_{h}} \vertiii{ \bsigma - \btau }_{Y^*}, \quad \bsigma \in W,
\end{equation*}
i.e., $W_{2,h}$ approximates $W$ always better than $W_h$.
Next, by comparing the two subdivisions $\cT_{1,h}$ and $\cT_{2,h}$, we observe that the characteristic mesh size of $\cT_{2,h}$ is smaller than that of $\cT_{1,h}$.
Moreover, $\cT_{2,h}$ has better shape-regularity than $\cT_{1,h}$ in general; see~\cite[Definition~4.4.13]{BS:2008} for the definition of the shape-regularity.
As the approximability of a mesh depends on the characteristic mesh size and the shape-regularity~\cite{BS:2008}, we can expect that $W_{2,h}$ approximates $W$ better than $W_{1,h}$.
Comparing Corollaries~\ref{Cor:FEM},~\ref{Cor:S_FEM}, and~\ref{Cor:SSE}, we conclude that the discretization error of the SSE method is less than that of the others.
In the next section, we will present numerical results that verify the superior approximability of $W_{2,h}$ compared to $W_h$ and $W_{1,h}$.

We conclude this section by presenting a convergence rate analysis of the SSE method.
As shown in Theorem~\ref{Thm:rate}, the $O(h)$ convergence of the strain error is guaranteed for the SSE method.

\begin{theorem}
\label{Thm:rate}
Let $\bu \in V$ and $\bar{\bu}_h \in V_h$ solve~\eqref{weak} and~\eqref{eg_SSE}, respectively.
Under the assumption $\bu \in (H^2 (\Omega))^2$, we have
\begin{equation*}
\vertiii{ \bepsilon [\bu] - \bar{\bepsilon} [\bar{\bu}_h] }_{Y} \leq Ch,
\end{equation*}
where 
\begin{equation*}
\bepsilon[\bv] = \bB \bv, \gap \gap \bar{\bepsilon}[\bv] = P_{2,h}P_{1,h}(\bB \bv), \quad \bv \in V_h,
\end{equation*}
and $C$ is a positive constant independent of $h$.
\end{theorem}
\begin{proof}
Throughout this proof, for two positive real numbers $A$ and $B$ depending on the parameter $h$, we write $A \lesssim B$ if there exists a positive constant $C$ independent of $h$ such that $A \leq C B$.
Since $\bepsilon [V_h]$ is a closed subspace of $W$, by~\cite[Theorem~1.7]{Teschl:2009}, there exists $\bv^{\dag} \in V_h$ such that
\begin{equation}
\label{v_proj}
\inf_{\bv \in V_h} \vertiii{\bepsilon [\bu] - \bepsilon [\bv]}_Y = \vertiii{\bepsilon [\bu] - \bepsilon [\bv^{\dag}]}_Y.
\end{equation}
Using the triangular inequality and~\eqref{v_proj}, we have
\begin{equation} \begin{split}
\label{term_splitting}
\vertiii{\bepsilon [\bu] - \bar{\bepsilon} [\bv^{\dag}]}_Y &\leq 
\vertiii{\bepsilon [\bu] - P_{2,h}\bepsilon [\bu]}_Y + 
\vertiii{P_{2,h} \left( \bepsilon [\bu] - P_{1,h} \bepsilon [\bu] \right)}_Y + 
\vertiii{P_{2,h} P_{1,h} \left( \bepsilon [\bu] - \bepsilon [\bv^{\dag}] \right)}_Y \\
&\leq \vertiii{\bepsilon [\bu] - P_{2,h}\bepsilon [\bu]}_Y + 
\vertiii{\bepsilon [\bu] - P_{1,h} \bepsilon [\bu] }_Y + 
\vertiii{\bepsilon [\bu] - \bepsilon [\bv^{\dag}] }_Y \\
&= \inf_{\bdelta_2 \in W_{2,h}} \vertiii{\bepsilon[\bu] - \bdelta_2}_Y
+ \inf_{\bdelta_1 \in W_{1,h}} \vertiii{\bepsilon[\bu] - \bdelta_1}_Y
+ \inf_{\bv \in V_h} \vertiii{\bepsilon [\bu] - \bepsilon [\bv]}_Y.
\end{split}
\end{equation}
Meanwhile, we have $\bepsilon [\bu]$ and $\bsigma [\bu] = \bD \bB \bu$ belong to $(H^1 (\Omega))^3$ since $\bu \in (H^2 (\Omega)^2$.
Recalling some standard results from the polynomial approximation theory in Sobolev spaces~\cite[Chapter~4]{BS:2008}, we get
\begin{equation} \begin{split}
\label{estimates}
\inf_{\bv \in V_h} \vertiii{\bepsilon [\bu] - \bepsilon [\bv]}_Y \lesssim h, \quad
\inf_{\bdelta_1 \in W_{1,h}} \vertiii{\bepsilon [\bu] - \bdelta_1}_Y \lesssim h, \quad
\inf_{\bdelta_2 \in W_{2,h}} \vertiii{\bepsilon [\bu] - \bdelta_2}_Y \lesssim h, \\
\inf_{\btau_1 \in W_{1,h}} \vertiii{\bsigma [\bu] - \btau_1}_{Y^*} \lesssim h, \quad
\inf_{\btau_2 \in W_{2,h}} \vertiii{\bsigma [\bu] - \btau_2}_{Y^*} \lesssim h.
\end{split} \end{equation}
Combining Corollary~\ref{Cor:SSE}, \eqref{term_splitting}, and~\eqref{estimates} yields the desired result.
\end{proof}

\section{Numerical experiments}
\label{Sec:Numerical}
In this section, numerical experiments are conducted to support the theoretical results presented in the previous sections.  
The strain-smoothed elements pass three basic numerical tests: the zero energy mode, isotropic element, and patch tests; see~\cite{Bathe:1996,LL:2018,LKL:2021}.

\begin{figure}[]
\centering
\includegraphics[width=0.8\textwidth]{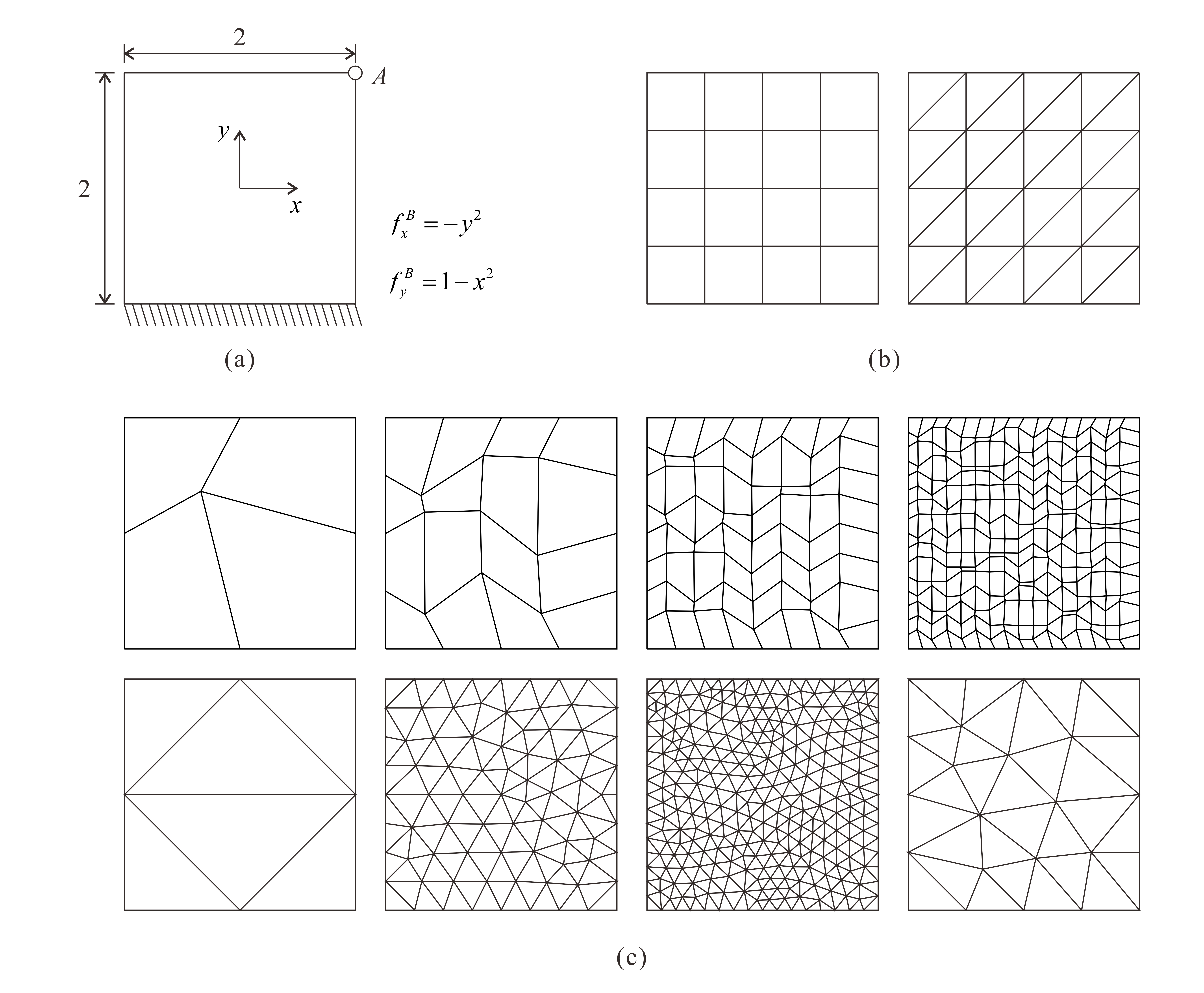}
\caption{The block problem: (a) Problem description (plane stress condition, $E=1 \times 10^3$ and $\nu = 0.2$). (b) Regular meshes of triangular and quadrilateral elements when $N=4$. (c) Distorted meshes of triangular and quadrilateral elements.}
\label{Fig:block}
\end{figure}

We consider the simple block problem shown in Fig.~\ref{Fig:block}. 
The block is subjected to body forces $f^B_x = -y^2$ and $f^B_y = 1-x^2$, and the clamped boundary condition is applied along the bottom edge.  
The plane stress condition is assumed, and the material properties are given as Young's modulus $E = 1 \times 10^3$ and Poisson's ratio $\nu = 0.2$.

For domain discretization, we use regular meshes of $N \times N$ triangular and quadrilateral elements~($N=2$, $4$, $8$, and $16$) shown in Fig.~\ref{Fig:block}(b).
In addition, distorted meshes of quadrilateral elements are constructed by repositioning the internal nodes of the regular meshes.
The distorted meshes of triangular elements are obtained using the commercial software ANSYS with the total number of elements $N_{e}=6$, $32$, $128$, and $500$, as shown in Fig.~\ref{Fig:block}(c).

First, the approximabilities of $W_{h}$, $W_{1,h}$, and $W_{2,h}$ are compared by measuring the discretization errors
\begin{equation}
 \label{proj_error}
 \inf_{\bdelta \in S_h} \vertiii{\bepsilon_{\mathrm{ref}} - \bdelta}
 = \vertiii{\bepsilon_{\mathrm{ref}} - \proj_{S_h} \bepsilon_{\mathrm{ref}}}, \quad S_h = W_h, W_{1,h}, W_{2,h},
\end{equation}
that occur when projecting the reference strain $\bepsilon_{\mathrm{ref}} \in W$, where the subscript ``ref'' denotes the reference finite element solution.
A $64 \times 64$ mesh of 9-node quadrilateral elements is used to calculate the reference strain.
Tables~\ref{Table:T3} and~\ref{Table:Q4} provide the discretization errors for the triangualr and quadrilateral meshes, respectively.
The approximabilities of $W_h$ and $W_{1,h}$ seem comparable to each other.
This is because the approximability of a mesh relies on both the mesh size and the shape of the mesh; while $\cT_{1,h}$ has a relatively small mesh size than $\cT_h$, its shape-regularity is worse owing to its longish shape.
On the other hand, as expected in Sect.~\ref{Sec:Convergence}, one can observe that $W_{2,h}$ exhibits the highest accuracy under all conditions.
This observation supports our claim that the SSE method is more convergent than other algorithms.

\begin{table} \centering
\begin{tabular}{c c c c c} \hline
{} & $N \times N$ (or $N_{e}$) & $W_{h}$ & $W_{1,h}$ & $W_{2,h}$ \\
\hline
\multirow{4}{*}{\begin{tabular}{c}Regular\\mesh\end{tabular}}
& $2 \times 2$ & $1.538 \times 10^{-3}$ & $1.182 \times 10^{-3}$ & $9.545 \times 10^{-4}$ \\
& $4 \times 4$ & $9.042 \times 10^{-4}$ & $6.943 \times 10^{-4}$ & $5.080 \times 10^{-4}$ \\
& $8 \times 8$ & $4.742 \times 10^{-4}$ & $3.719 \times 10^{-4}$ & $2.601 \times 10^{-4}$ \\
& $16 \times 16$ & $2.416 \times 10^{-4}$ & $1.917 \times 10^{-4}$ & $1.313 \times 10^{-4}$ \\
\hline
\multirow{4}{*}{\begin{tabular}{c}Distorted\\mesh\end{tabular}}
& 6 & $1.734 \times 10^{-3}$ & $1.608 \times 10^{-3}$ & $1.167 \times 10^{-3}$ \\
& 32 & $9.535 \times 10^{-4}$ & $9.273 \times 10^{-4}$ & $5.314 \times 10^{-4}$ \\
& 128 & $4.914 \times 10^{-4}$ & $4.888 \times 10^{-4}$ & $2.694 \times 10^{-4}$ \\
& 500 & $2.536 \times 10^{-4}$ & $2.588 \times 10^{-4}$ & $1.379 \times 10^{-4}$ \\
\hline
\end{tabular}
\caption{Discretization errors~\eqref{proj_error} for the clamped block problem when using regular and distorted triangular meshes.}
\label{Table:T3}
\end{table}

\begin{table} \centering
\begin{tabular}{c c c c c} \hline
{} & $N \times N$ & $W_{h}$ & $W_{1,h}$ & $W_{2,h}$ \\
\hline
\multirow{4}{*}{\begin{tabular}{c}Regular\\mesh\end{tabular}}
& $2 \times 2$ & $2.009 \times 10^{-3}$ & $1.818 \times 10^{-3}$ & $1.174 \times 10^{-3}$ \\
& $4 \times 4$ & $1.174 \times 10^{-3}$ & $1.208 \times 10^{-3}$ & $6.189 \times 10^{-4}$ \\
& $8 \times 8$ & $6.189 \times 10^{-4}$ & $7.027 \times 10^{-4}$ & $3.169 \times 10^{-4}$ \\
& $16 \times 16$ & $3.169 \times 10^{-4}$ & $3.819 \times 10^{-4}$ & $1.595 \times 10^{-4}$ \\
\hline
\multirow{4}{*}{\begin{tabular}{c}Distorted\\mesh\end{tabular}}
& $2 \times 2$ & $1.842 \times 10^{-3}$ & $1.809 \times 10^{-3}$ & $1.170 \times 10^{-3}$ \\
& $4 \times 4$ & $1.187 \times 10^{-3}$ & $1.381 \times 10^{-3}$ & $6.436 \times 10^{-4}$ \\
& $8 \times 8$ & $6.759 \times 10^{-4}$ & $9.877 \times 10^{-4}$ & $3.510 \times 10^{-4}$ \\
& $16 \times 16$ & $3.463 \times 10^{-4}$ & $8.912 \times 10^{-4}$ & $1.768 \times 10^{-4}$ \\
\hline
\end{tabular}
\caption{Discretization errors~\eqref{proj_error} for the clamped block problem when using regular and distorted quadrilateral meshes.}
\label{Table:Q4}
\end{table}

We demonstrate the accuracy and convergence behavior of the strain-smoothed 3-node triangular element~(SSE T3) and 4-node quadrilateral element~(SSE Q4). 
Three 3-node triangular elements and four 4-node quadrilateral elements are considered for comparison: the standard 3-node element~(FEM T3), the node-based 3-node element~(NS-FEM T3), the edge-based 3-node element~(ES-FEM T3), the 4-node element using piecewise linear shape functions~(FEM PL-Q4), the 4-node element using bilinear shape functions~(FEM BL-Q4), the cell-based 4-node element subdividing the element into four quadrilateral smoothing cells~(CS-FEM Q4), and the edge-based 4-node element~(ES-FEM Q4).
Note that the FEM PL-Q4 element, not the FEM BL-Q4 element, corresponds to the standard finite element method described in Sect.~\ref{Subsec:StdFEM}.

In the following, we write
\begin{equation*}
\bepsilon_h = \begin{cases} \bepsilon[ \bu_h]  & \textrm{ for FEM,} \\
\hat{\bepsilon} [\hat{\bu}_h] & \textrm{ for ES-FEM,} \\
\bar{\bepsilon} [\bar{\bu}_h] & \textrm{ for SSE;}
\end{cases}
\end{equation*}
see Corollaries~\ref{Cor:FEM}--\ref{Cor:SSE} for the notations.

To evaluate the accuracy and convergence behavior of the numerical solutions, we use the relative error in the strain energy norm:
\begin{equation}
\label{Ee}
E_e = \frac{\vertiii{\bepsilon_h - \bepsilon_{\mathrm{ref}}}_Y}{\vertiii{\bepsilon_{\mathrm{ref}}}_Y}.
\end{equation}

\begin{figure}[]
\centering
\includegraphics[width=0.9\textwidth]{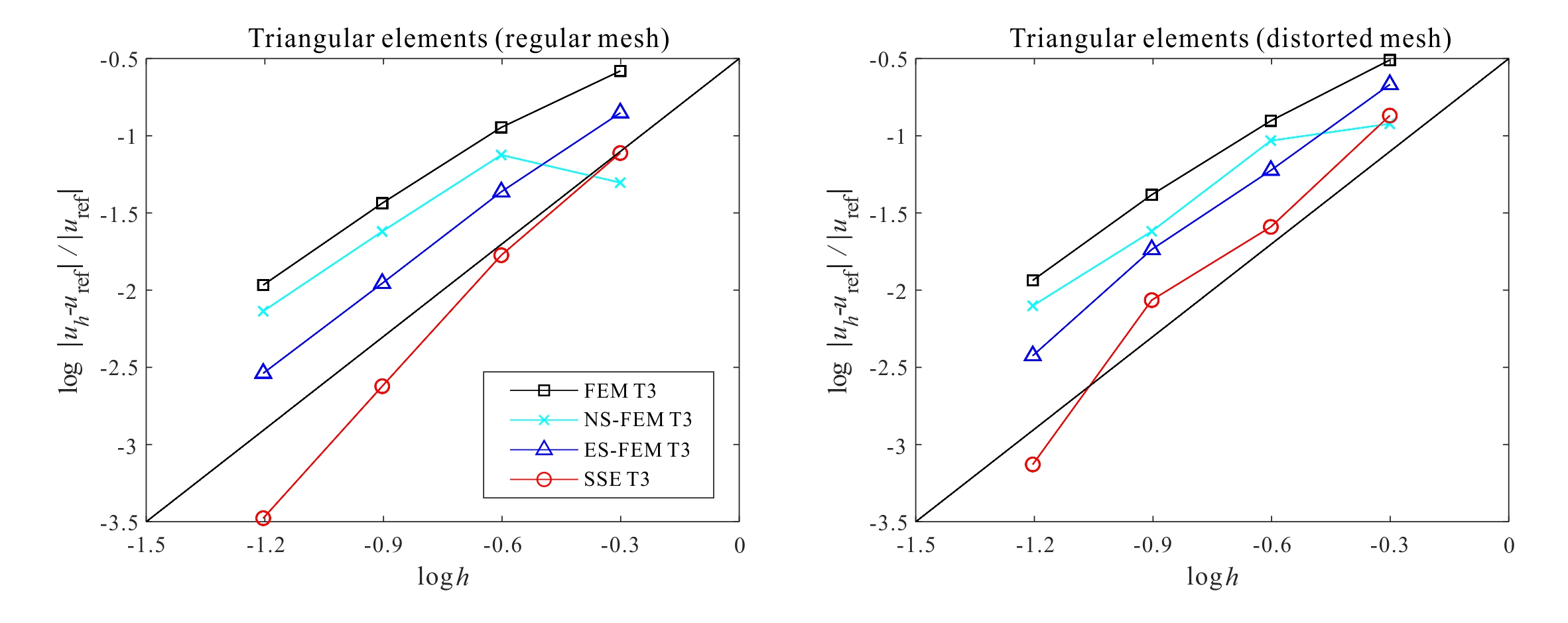}
\caption{Convergence curves for the relative error in the horizontal displacement at point $A$ for the triangular elements.}
\label{Fig:block_disp_t3}
\end{figure}

\begin{figure}[]
\centering
\includegraphics[width=0.9\textwidth]{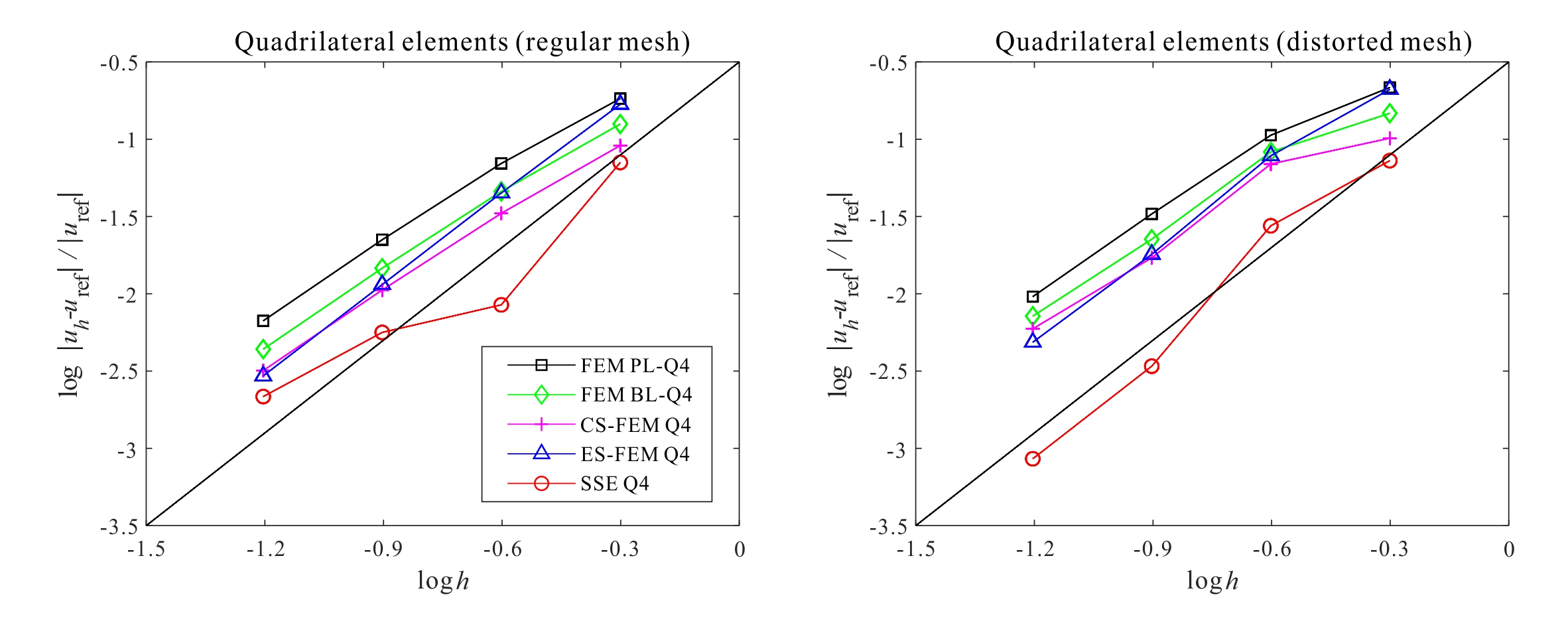}
\caption{Convergence curves for the relative error in the horizontal displacement at point $A$ for the quadrilateral elements.}
\label{Fig:block_disp_q4}
\end{figure}

\begin{figure}[]
\centering
\includegraphics[width=0.9\textwidth]{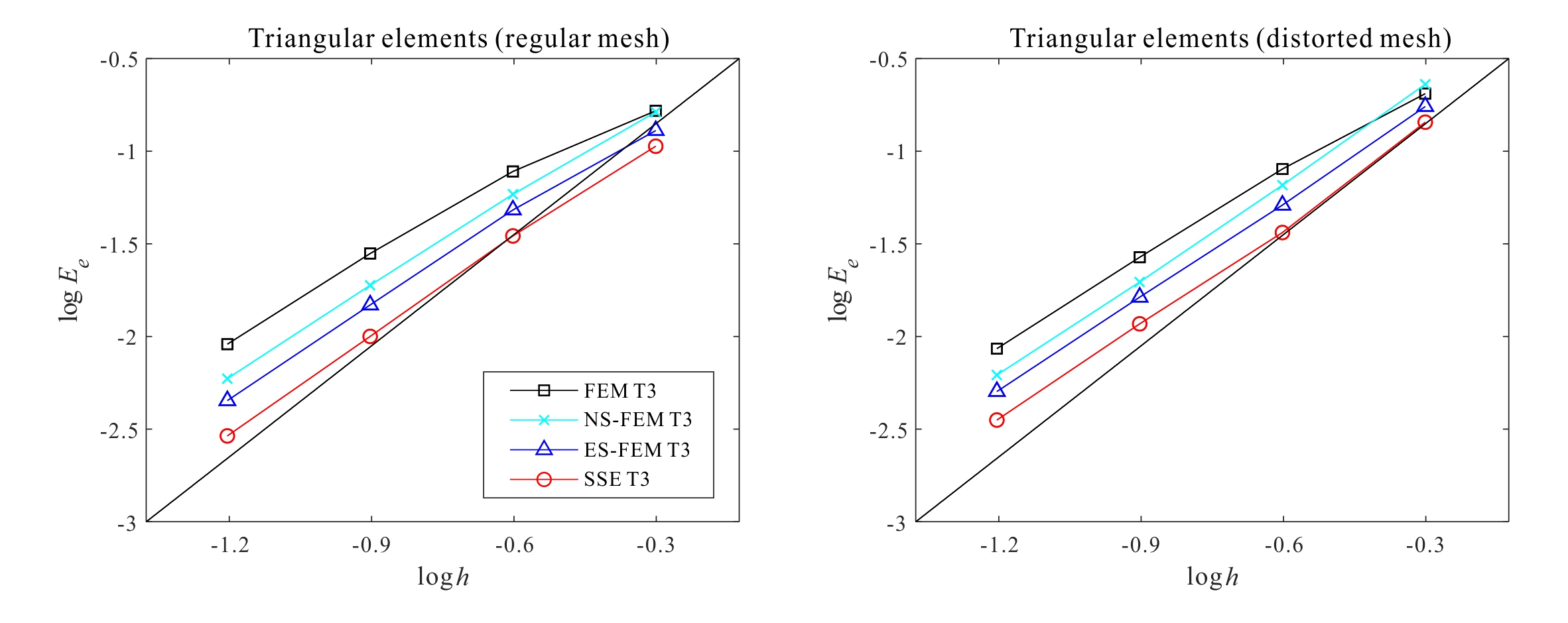}
\caption{Convergence curves for the relative error in energy norm~\eqref{Ee} for the triangular elements. The diagonal line denotes the optimal convergence rate.}
\label{Fig:block_norm_t3}
\end{figure}

\begin{figure}[]
\centering
\includegraphics[width=0.9\textwidth]{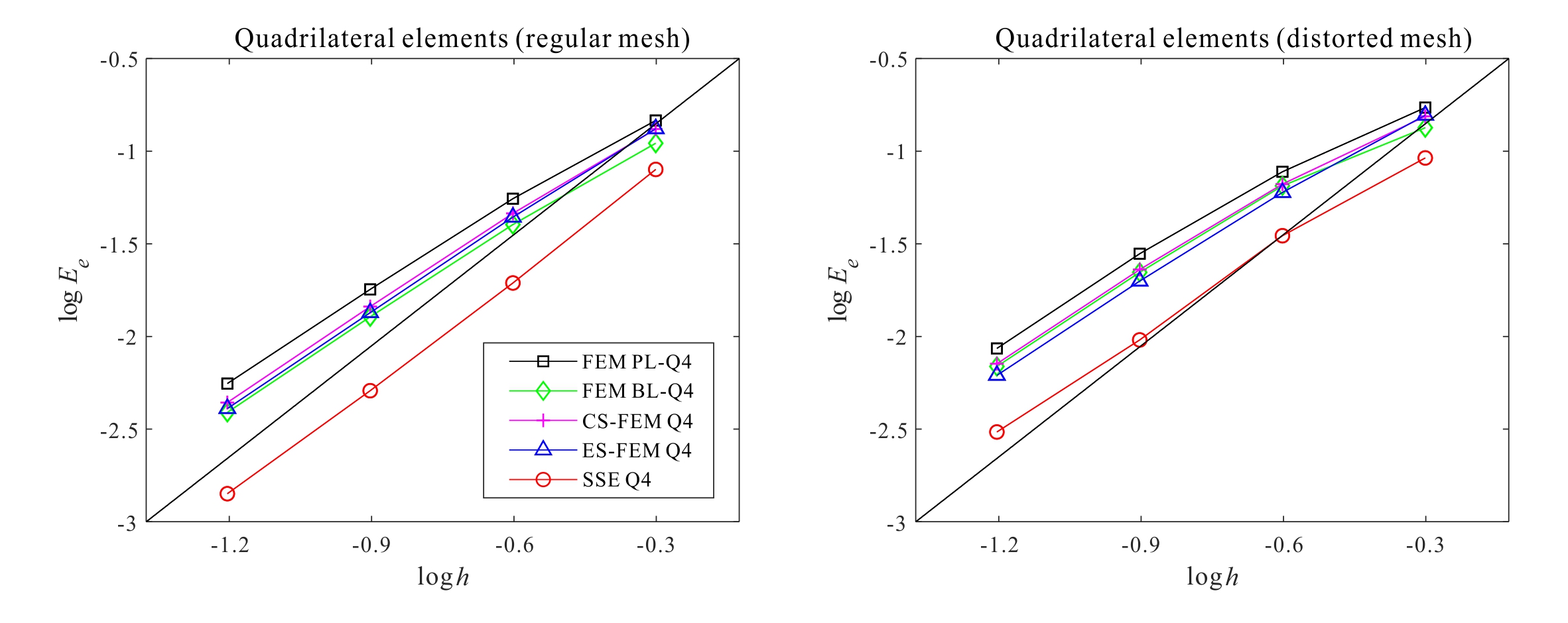}
\caption{Convergence curves for the relative error in energy norm~\eqref{Ee} for the quadrilateral elements. The diagonal line denotes the optimal convergence rate.}
\label{Fig:block_norm_q4}
\end{figure}

Figs.~\ref{Fig:block_disp_t3} and ~\ref{Fig:block_disp_q4} show the convergence curves for the relative error in the horizontal displacement at point $A$ for the triangular and quadrilateral elements, respectively. 
Figs.~\ref{Fig:block_norm_t3} and ~\ref{Fig:block_norm_q4} illustrate the convergence curves obtained using the relative error in energy norm~\eqref{Ee} for the triangular and quadrilateral elements, respectively. 
In the convergence curves, the element size $h$ is defined as $h=1/N$. The equivalent values of $N$ for the distorted meshes of the triangular elements are obtained by $N=\sqrt{N_{e}/2}$.
The reference solution is obtained using a $64 \times 64$ mesh of 9-node quadrilateral elements.
The results for the triangular elements show that the SSE T3 element shows the best accuracy, followed by the ES-FEM T3 and NS-FEM T3 elements, and the FEM T3 element provides the lowest accuracy.
The SSE Q4 element also shows much better convergence behavior compared with the other quadrilateral elements in this problem.
This well supports the theoretical investigations presented in Sect.~\ref{Sec:Convergence}. 

In~\cite{LL:2018,LL:2019,LKL:2021}, the performance of the elements adopting the SSE method was verified using various numerical examples. The accuracy was evaluated by comparing the displacement, strain, and stress parameters, and the computational efficiency was estimated. The studies also considered the effect of mesh distortion. The geometric nonlinear extension of the strain-smoothed elements is discussed in~\cite{LKL:2021}.

\section{Conclusion}
\label{Sec:Conclusion}
In this paper, we presented a novel mixed variational principle that provides a unified view of the standard FEM, the S-FEM, and the SSE method.
The proposed variational principle naturally generalizes the Hu--Washizu variational principle, and the SSE method can be derived as a conforming Galerkin approximation of the proposed variational principle.
Therefore, invoking the standard theory of mixed FEMs yielded a unified convergence analysis for the SSE method and other existing FEMs with strain smoothing.
In addition, our analysis explains why the SSE method demonstrates improved performance compared to other methods.
Our theoretical results on the improved performance of the SSE method were verified through numerical experiments.

There are a few interesting topics for future works.
Although the convergence of the SSE method was guaranteed by Corollary~\ref{Cor:SSE}, a sharp and rigorous estimate of the convergence rate of the method remains open.
We also note that generalizing the proposed variational principle to apply the SSE method in three dimensions is not straightforward~\cite[Section~3]{LL:2018}.
Finally, SSE methods for nonlinear elliptic partial differential equations and their corresponding variational principles will be considered in future research.

\section*{Acknowledgement}
This work was initially started with the help of Professor Phill-Seung Lee through a meeting on the mathematical background of the SSE method.
The authors would like to thank him for his insightful comments and assistance.

\appendix
\section{Abstract convergence theory of nonconforming finite element methods}
\label{App:FEM}
In this appendix, we present an abstract convergence theory of nonconforming Galerkin methods.
Let $H$ be a Hilbert space and let $V$ and $V_h$ be subspaces of $H$ such that $V_h \not\subset V$.
Assume that $|\cdot|_H$ is a seminorm on $H$ such that $|\cdot|_H$ is positive definite on $V \cup V_h$, i.e.,
\begin{equation*}
|u|_H > 0, \quad u \in (V \cup V_h)\setminus \{ 0 \}.
\end{equation*}
Let $a(\cdot, \cdot)$:~$H \times H \rightarrow \mathbb{R}$ be a blinear form on $H$ which is continuous and coercive with respect to $|\cdot|_H$, i.e., there exist two positive constants $C$ and $\alpha$ satisfying
\begin{align*}
a(u,v) &\leq C | u |_H | v |_H,\\
a(u,u) &\geq \alpha |u|_H^2
\end{align*}
for $u,v\in H$.
In Theorem~\ref{Thm:nonconforming}, we present an error estimate for the variational problem
\begin{equation}
\label{App:weak}
a(u,v) = f(v), \quad v\in V
\end{equation}
with respect to its nonconforming Galerkin approximation
\begin{equation}
\label{App:Galerkin}
a(u_h, v) = f(v), \quad v \in V_h,
\end{equation}
where $f \in H^*$.

\begin{theorem}
\label{Thm:nonconforming}
Let $u \in V$ and $u_h \in V_h$ solve~\eqref{App:weak} and~\eqref{App:Galerkin}, respectively.
Then we have
\begin{equation*}
| u - u_h |_H \leq \left( 1+ \frac{C}{\alpha} \right) \inf_{v \in V_h} |u-v|_H + \frac{1}{\alpha} \sup_{w \in V_h \setminus \{ 0\}} \frac{|a(u-u_h , w)|}{|w|_H}.
\end{equation*}
\end{theorem}
\begin{proof}
One can easily obtain the desired result by following the argument in~\cite[Lemma~10.1.1]{BS:2008}.
\end{proof}

Note that Theorem~\ref{Thm:nonconforming} is written in terms of seminorm $|\cdot|_H$ while the existing standard results~(see, e.g.,~\cite{BS:2008,Ciarlet:2002}) are written in terms of norm.
In this sense, Theorem~\ref{Thm:nonconforming} is a generalization of the standard results. 

\bibliographystyle{elsarticle-num-names} 
\bibliography{refs_SSE}

\end{document}